\documentclass[english]{article}
\usepackage[T1]{fontenc}
\usepackage[latin9]{inputenc}
\usepackage{geometry}
\usepackage{xcolor}
\usepackage{babel}
\usepackage{verbatim}
\usepackage{prettyref}
\usepackage{float}
\usepackage{mathtools}
\usepackage{algorithm2e}
\usepackage{amsmath}
\usepackage{amsthm}
\usepackage{amssymb}
\usepackage{graphicx}
\usepackage{subfig}
\usepackage[authoryear]{natbib}
\usepackage[unicode=true,pdfusetitle,
 bookmarks=true,bookmarksnumbered=false,bookmarksopen=false,
 breaklinks=true,pdfborder={0 0 1},backref=false,colorlinks=true]
 {hyperref}
\hypersetup{
 linkcolor=blue,citecolor=blue,urlcolor=blue,filecolor=blue,pdfpagelayout=OneColumn,pdfnewwindow=true,pdfstartview=XYZ,plainpages=false,pdfpagelabels,hyperindex=true}

\makeatletter

\floatstyle{ruled}
\newfloat{algorithm}{tbp}{loa}
\providecommand{\algorithmname}{Algorithm}
\floatname{algorithm}{\protect\algorithmname}

\theoremstyle{definition}
    \ifx\thechapter\undefined
      \newtheorem{defn}{\protect\definitionname}
    \else
      \newtheorem{defn}{\protect\definitionname}[chapter]
    \fi
\theoremstyle{plain}
    \ifx\thechapter\undefined
      \newtheorem{assumption}{\protect\assumptionname}
    \else
      \newtheorem{assumption}{\protect\assumptionname}[chapter]
    \fi
\theoremstyle{plain}
    \ifx\thechapter\undefined
      \newtheorem{prop}{\protect\propositionname}
    \else
      \newtheorem{prop}{\protect\propositionname}[chapter]
    \fi
\theoremstyle{plain}
    \ifx\thechapter\undefined
	    \newtheorem{thm}{\protect\theoremname}
	  \else
      \newtheorem{thm}{\protect\theoremname}[chapter]
    \fi
\theoremstyle{plain}
    \ifx\thechapter\undefined
  \newtheorem{cor}{\protect\corollaryname}
\else
      \newtheorem{cor}{\protect\corollaryname}[chapter]
    \fi

\@ifundefined{date}{}{\date{}}
\usepackage{babel}

\usepackage{stackengine}
\usepackage{bbm}
\usepackage{enumitem}

\usepackage{authblk}

\setlist{leftmargin=*, topsep=0.5em, parsep=0pt, itemsep=1em, labelindent=0pt, align=left}

\@ifundefined{showcaptionsetup}{}{%
 \PassOptionsToPackage{caption=false}{subfig}}
\usepackage{subfig}
\makeatother

\providecommand{\assumptionname}{Assumption}
\providecommand{\corollaryname}{Corollary}
\providecommand{\definitionname}{Definition}
\providecommand{\propositionname}{Proposition}

\providecommand{\theoremname}{Theorem}

\providecommand{\keywords}[1]{\textit{Keywords:} #1}
\providecommand{\AMS}[1]{AMS subject classifications: #1}

\usepackage{nameref}

\begin{document}
\title{Robust utility maximization under model uncertainty via a penalization
approach}
\author[1,2]{Ivan Guo}
\author[3]{Nicolas Langren\'e}
\author[1,2,4]{Gr\'egoire Loeper}
\author[1]{Wei Ning}

\affil[1]{\small School of Mathematical Sciences, Monash University, Melbourne, Australia}
\affil[2]{\small Centre for Quantitative Finance and Investment Strategies, Monash University, Australia}
\affil[3]{\small Data61, Commonwealth Scientific and Industrial Research Organisation, RiskLab Australia}
\affil[4]{\small BNP Paribas Global Markets}

\date{First version: July 31, 2019\\
This revised version: July 3, 2020}

\maketitle
\begin{abstract}
This paper addresses the problem of utility maximization under uncertain parameters. In contrast with the classical  approach, where the parameters of the model evolve freely within
a given range, we constrain them via a penalty function. We show that this
robust optimization process can be interpreted as a two-player zero-sum
stochastic differential game. We prove that the value function satisfies
the Dynamic Programming Principle and that it is the unique viscosity
solution of an associated Hamilton--Jacobi--Bellman--Isaacs equation.
We test this robust algorithm on real market data.
The results show that robust portfolios generally have higher expected utilities and are more stable under strong market downturns.
To solve for the value function, we derive an analytical solution in the logarithmic utility case and
obtain accurate numerical approximations in the general case by three
methods: finite difference method, Monte Carlo simulation, and Generative Adversarial Networks.  
\end{abstract}

\AMS{49N90, 49K35, 49K20, 49L20, 49L25, 91G80}

\section{Introduction\label{sec:Introduction}}

This paper addresses the problem of continuous-time utility maximization.
Besides the choice of utility function, a key element in the formulation
of such a problem is the a~priori knowledge assumed for the evolution
of the underlying assets (e.g., the expected returns and the quadratic
covariation of the diffusion process). In a landmark paper, \citet{merton1969lifetime}
found an explicit solution for the problem of optimal portfolio selection
and consumption, for a constant relative risk aversion (CRRA) utility
function $\frac{X^{\gamma}}{\gamma}$, $\gamma\in(0,1)$ (a.k.a. power
utility or isoelastic utility). He found that the optimal fraction
of the wealth to be invested in the risky asset is given by $\pi^{*}=\frac{\mu-r}{\sigma^{2}(1-\gamma)}$\footnote{Here, $\mu$ is the expected rate of asset returns, $\sigma^{2}$
is the variance of the asset returns, $r$ is the risk-free interest
rate and $1-\gamma$ is the relative risk aversion constant.}, which is independent of both time and the current wealth, even though
this quantity is a~priori allowed to evolve dynamically. This conclusion
is arguably one of the most important results in portfolio optimization
(and it is also consistent with the results of Markowitz portfolio
optimization \citealt{markowitz1952portfolio}). It has led to various
extensions, some of which are illustrated in the textbook by \citet{rogers2013optimal}.

In the original Merton problem, the evolution of the risky asset,
although stochastic by essence, is governed by the Black-Scholes model
\citep{black1973pricing} with fixed parameters $\mu,r$ and $\sigma$.
This is a very simplistic model for the underlying asset price. Stochastic
models (for the volatility and interest rates) that describe the price
evolution more realistically have later emerged. Several papers have
addressed the problem in this context: \citet{matoussi2015robust}
examined the case of stochastic volatility, while \citet{noh2011optimal}
addressed the case of stochastic interest rates. The expected return
(or drift) $\mu$ plays an essential role in the optimal allocation;
even when it is considered stochastic, it is still assumed to be an
\textit{observable} input of the problem. This assumption clearly
does not match the reality that investors are facing. Several works
by \citet{lakner1995utility} and then \citet{bel2017forecasting}
addressed the utility maximization problem with an uncertain drift,
although it was assumed to follow some form of prescribed dynamics
or prior distribution.

Two decades ago, the concept of \textit{robust portfolio optimization}
had emerged. It was first introduced in the operations research literature
by \citet{el1997robust} and \citet{ben1998robust}. Instead of assuming
a model with a known drift, interest rate or volatility, the problem
of \textit{robust} optimal allocation assumes that they will evolve
dynamically in the most unfavourable way within a given range. The
resulting allocation process tends to be more stable and less vulnerable
to changes and misspecifications in model parameters.

There has been a substantial amount of literature on robust portfolio
optimization over the last decade and the area is still developing.
A comprehensive introduction of the trends and methods can be found
in the book by \citet{fabozzi2007robust}. \citet{gabrel2014recent}
provided an overview of advances in robust optimization, including
but not limited to applications in finance, where they stated that
``robustifying'' stochastic optimization is one of the key advancements
that should develop following the 2007 financial crisis. We list below
a few pieces of influential research in this direction. For instance,
\citet{elliott2009robust} supposed that an agent wants to maximize
the minimal utility function, over a family of probability measures.
This problem was then formulated as a Markovian regime-switching model,
where the market parameters are modulated by a continuous-time finite-state
Markov chain that is determined by the probability measures. \citet{glasserman2013robust}
went beyond parameter uncertainties to consider the effect of changes
in the probability distributions that define an underlying model.
They used relative entropy to quantify the deviation of the worst-case
model from a baseline model. \citet{fouque2016portfolio} studied
an asset allocation problem with stochastic volatility and uncertain
correlation, and derived closed-form solutions for a class of utility
functions. \citet{ismail2019robust} studied a robust Markowitz portfolio
selection problem under covariance uncertainty. The value function
is obtained by optimizing the worst-case mean-variance functional,
over the admissible investing strategies $\alpha$. They then solved
this problem by the McKean-Vlasov dynamic programming approach and
characterized the solution with a Bellman-Isaacs PDE. They also illustrated
the robust efficient frontier in two examples: uncertain volatilities
and uncertain correlation. Last but not least, we also mention the
work by \citet{talay2002worst}, which studied the robust optimization
problem in the context of derivatives hedging.

A robust investment process can be interpreted as a two-player game.
On one hand, the market can be thought of as an adversarial player
controlling the volatility (or the drift) in order to minimize the
gains of an investor, on the other hand, the investor, who controls
the allocation of the portfolio, is trying to maximize her gains under
the worst possible behaviour of the market. The two controllers have
conflicting interests, with the gain of one player being a loss for
the other. Hence we call this competition between the investor and
the market a two-player zero-sum stochastic differential game (SDG).
Differential games were first introduced by \citet{isaacs1965differential};
the book by \citet{fleming2006controlled} provides a concise introduction
to the theory of viscosity solutions and deterministic zero-sum differential
games. The first complete theory for two-player zero-sum SDGs was
developed by \citet{fleming1989existence}, where they proved the
existence of value functions of the games. \citet{buckdahn2008stochastic}
generalized the results of \citet{fleming1989existence} by considering
the gain functional as a solution of a Backward Stochastic Differential
Equation (BSDE). With the help of BSDE methods, they proved the Dynamic
Programming Principle (DPP) for the value functions in a more straightforward
approach. Some more recent works on zero-sum SDG
include \citet{hernandez2018zero}, \citet{baltas2019robust} and \citet{cosso2019zero}.

The main novelty of our work is threefold. Firstly, we do not assume a given range
of parameters in the evolution of the underlying process. 
In other papers considering uncertain volatility, the authors assume the admissible $\sigma\in[\sigma_{\min},\sigma_{\max}]$,
where $\sigma_{\min}$ and $\sigma_{\max}$ are model bounds in
accordance with the uncertainty about future fluctuations.
Instead, we allow the parameters to move freely and use a penalty
function $F=F(r,\mu,\sigma,\hdots)$ to penalize unrealistic values
of the parameters. Mathematically speaking, the penalty function gives
some \textit{coercivity} to the problem so that an optimal solution
can be found. This approach has been used for robust derivatives pricing
in \citet{tan2013optimal} and \citet{guo2017local}. Note that one
can asymptotically recover the aforementioned approaches that involve
a \textit{fixed parameter range}, by taking the penalty function $F$
to be 0 over a given set and $+\infty$ outside.

Secondly, in the classical papers studying two-player zero-sum
SDGs, \citet{fleming1989existence} and \citet{nisio2015stochastic}
made the assumptions that the domain is bounded and the utility function
$U$ is bounded and Lipschitz continuous. 
The present paper extends these results to more general assumptions
by considering an unbounded domain and an unbounded utility function
$U$. Moreover, we prove that the lower- and upper-value of the SDG
\eqref{eq:dynamic of X}-\eqref{eq:gain function} in fact coincide.

Last but not least, we devise two innovative algorithms to compute the value functions, 
which are control randomization and Generative Adversarial Networks (GANs). 
In particular, it is, to our knowledge,
the first application of the control randomization method (see \citealt{kharroubi2014numerical})
in the context of a robust portfolio optimization problem. It is also the first time GANs are used to solve 
a robust optimization problem in the field of quantitative finance. 

GANs are an exciting recent innovation in machine learning. 
The fundamental principle of GANs is to use two different neural networks as two opponents with conflicting goals, and its solution is a Nash equilibrium. Hence, GANs training is closely related to game theory. \citet{cao2020connecting} reviewed the minimax structures underlying GANs, and they established theoretical connections between GANs and Mean-Field Games. However, there are few applications of GANs in quantitative finance so far. The only relevant work is by \citet{wiese2020quant}. Being inspired by GANs' ability to generate images,  they approximated a realistic asset price simulator using adversarial training techniques.

The rest of the paper is organized as follows. In \prettyref{sec:Problem-Formulation},
we formulate a portfolio optimization problem in a robust setting
and introduce the uncertain drift and uncertain volatility processes.
In the subsequent sections, we only focus on the uncertain volatility
case because the uncertain drift case can be solved in a similar way.
In \prettyref{sec:The-classical-results}, we define the value functions
for static games and two-player zero-sum SDGs. In \prettyref{sec:Existence-of-saddle}
we show that the differential game has a saddle point and as a consequence,
the lower- and upper-values of the SDG coincide. We prove that the
value function satisfies the DPP in \prettyref{sec:The-Dynamic-Programming}
and that our value function is the unique viscosity solution of an
HJBI equation in \prettyref{sec:Viscosity-solution-HJBI}. In section \prettyref{subsec:Analytical-Solution}, we derive a closed-form solution for the logarithmic utility. 
In section \prettyref{subsec:Monte Carlo comparison}, we add some noise to the covariance matrix and simulate portfolios with robust and non-robust strategies, respectively. Then, in section \prettyref{subsec:Empirical Experiment}, we test our robust mechanism by constructing two empirical portfolios using market data. 
In section \prettyref{subsec:Finite-Difference-Method} and  \prettyref{subsec:Monte-Carlo-Method}, we provide numerical results
for general utility functions using PDE techniques via finite difference
methods  and Monte Carlo simulations via control randomization. Finally, in section \prettyref{subsec:GANs}, we present the algorithm and result of solving a robust portfolio optimization problem with GANs. 

\section{Problem formulation\label{sec:Problem-Formulation}}

We consider a portfolio with $d$ risky assets and one risk-free asset
compounding at a constant interest rate $r\in\mathbb{R}$. The price
process of the risky assets is denoted by $S_{t}\in\mathbb{R}^{d}$
$(0\leq t\leq T)$, and the $i$th element of $S_{t}$ follows the
dynamics 
\begin{equation}
\frac{dS_{t}^{i}}{S_{t}^{i}}=\mu_{t}^{i}dt+\sum_{j=1}^{d}\sigma_{t}^{ij}dW_{t}^{j},\quad1\leq i\leq d,\label{eq:dynamic of S}
\end{equation}
with drift $\mu_{t}\in\mathbb{R}^{d}$, covariance matrix $\Sigma_{t}\in\mathbb{R}^{d\times d}$
and its square-root matrix $\sigma_{t}\coloneqq\Sigma_{t}^{\frac{1}{2}}\in\mathbb{R}^{d\times d}$.

We consider a  probability space
$\left(\Omega,\mathcal{F},\mathbb{P}\right)$, and processes $\mu,\Sigma$ which are progressively measurable with respect to the $\mathbb{P}$-augmented filtration of the $d$-dimensional Brownian motion $W_t$

Let $X_{t}\in\mathbb{R}$ be the value of the portfolio at time $t$.
A portfolio allocation strategy $\alpha_{t}\in\mathbb{R}^{d}$ represents
the proportion of total wealth the agent invests in the $d$ risky
assets at time $t$, and $1-\sum_{i=1}^{d}\alpha_{t}^{i}$ is the
proportion invested in the risk-free asset.

Assuming the strategy is self-financed, the wealth process evolves
as follows 
\[
\frac{dX_{t}}{X_{t}}=\sum_{i=1}^{d}\alpha_{t}^{i}\frac{dS_{t}^{i}}{S_{t}^{i}}+\left(1-\sum_{i=1}^{d}\alpha_{t}^{i}\right)rdt.
\]
We define $\mathbf{r}\coloneqq r\times\mathbb{\mathbf{1}}$ with $\mathbf{1}\in\mathbb{R}^{d}$
being a $d$-dimensional ones vector. The wealth evolution can be
rewritten as 
\begin{equation}
dX_{t}=X_{t}(\alpha_{t}^{\intercal}(\mu_{t}-\mathbf{r})+r)dt+X_{t}\alpha_{t}^{\intercal}\sigma_{t}dW_{t}.\label{eq:dynamic of X}
\end{equation}

We will follow the framework set in \citet{fleming1989existence}
and \citet{talay2002worst}.  We first introduce the canonical sample spaces for the underlying Brownian motion in  \eqref{eq:dynamic of S} and  \eqref{eq:dynamic of X}.
For each $t\in[0,T]$, we set
\[
\Omega_t \coloneqq(\omega\in C([t,T];\mathbb{R}^{d}):\omega_{t}=0).
\]
We denote by $\mathbb{F} = \mathcal{F}_{t,s}$ $ (s \in [t,T]) $, the filtration generated by the canonical process from time $t$ to time $s$. Equipped with the Wiener measure $\mathbb{P}_t$ on $\mathcal{F}_{t,T}$, the filtered probability space $(\Omega_t, \mathcal{F}_{t,T}, \mathbb{P}_t, \mathbb{F} )$ is the canonical sample space, and $W$ is the standard $d$-dimensional Brownian motion.

Now, we introduce the concept of \textit{admissible controls}. 
\begin{defn}
\label{def: An-admissible-vol}An admissible control process $\Sigma$
(resp. $\mu$) for the market on $[t,T]$ is a progressively measurable
process with respect to $\mathbb{F}$, taking values in a compact
convex set $B\subset\mathbb{S}^{d}$ (resp. $M\subset\mathbb{R}^{d}$),
where $\mathbb{S}^{d}\subset\mathbb{R}^{d\times d}$ is a set of symmetric
positive semi-definite matrices. The set of all admissible $\Sigma$
(resp. $\mu$) on $[t,T]$ is compact and convex, denoted by $\mathcal{B}$
(resp. $\mathcal{M}$). 
\end{defn}
\begin{defn}
\label{def: An-admissible-alpha}An admissible control process $\alpha$
for the investor on $[t,T]$ is a progressively measurable process
with respect to $\mathbb{F}$, taking values in a compact convex set
$A\subset\mathbb{R}^{d}$. The set of all admissible $\alpha$ is
compact and convex, denoted by $\mathcal{A}$.
\end{defn}
Note that although the sets for the value of the controls
are compact, in practice, $A=[-R,R]^{d},B=[-R,R]^{d\times d}\cap\mathbb{S}^{d}$
where $R$ is arbitrarily large.

Next, let us define the payoff function as the expectation of a terminal utility
function $U$ plus a penalty function $F$: 
\begin{equation}
J(t,x,\alpha,\mu,\Sigma)=\mathbb{E}^{t,x}\left[U(X_{T}^{\alpha,\mu,{\scriptscriptstyle \Sigma}})+\lambda_{0}\int_{t}^{T}F(\mu_{s},\Sigma_{s})ds\right],\label{eq:gain function}
\end{equation}
where $\mathbb{E}^{t,x}(\cdot)$ denotes
the expectation given the initial time and wealth $(t,x)\in[0,T]\times\mathbb{R}$.
and $\lambda_{0}\in\mathbb{R}$ is a positive constant. Throughout the paper, we will often include $\alpha,\mu$ and $\Sigma$
in the superscript of $X$ to indicate the dependency of the wealth
process on the allocation, drift and volatility processes. Our objective
is to find the optimal portfolio allocation process $\alpha$ that
maximizes the worst-case payoff function given by the drift process
$\mu$ or the covariance process $\Sigma$. Throughout the paper,
$F$ will be a convex function in $\Sigma_{s}$ and $\mu_{s}$. 

\subsection{Robust value functions\label{subsec:Uncertain-Volatility}}

We are now ready to define the value functions. In our problem, the
covariance (or drift) is unknown. We want to find the optimal portfolio
allocation process that maximizes the worst-case situation given by
the covariance (or drift). Then, given an initial condition $(t,x)\in[0,T]\times\mathbb{R}$,
this value is given by 
\[
\underline{u}(t,x)=\sup_{\alpha\in\mathcal{A}}\inf_{\Sigma\in\mathcal{B},\mu\in\mathcal{M}}\left\{ \mathbb{{E}}^{t,x}\left[U(X_{T}^{\alpha,\mu,{\scriptscriptstyle \Sigma}})+\lambda_{0}\int_{t}^{T}F(\mu_{s},\Sigma_{s})ds\right]\right\} .
\]
We say $\hat{\alpha}$ and $\hat{\Sigma},\hat{\mu}$ are optimal controls
if $\underline{u}(t,x)=J(t,x,\hat{\alpha},\hat{\mu},\hat{\Sigma})=\inf_{\Sigma\in\mathcal{B},\mu\in\mathcal{M}}J(t,x,\hat{\alpha},\mu,\Sigma)$.
Hereafter, we focus on the robust optimization problem with an uncertain
covariance, that is, 
\begin{equation}
\underline{u}(t,x)=\adjustlimits\sup_{\alpha\in\mathcal{A}}\inf_{\Sigma\in\mathcal{B}}\left\{ \mathbb{{E}}^{t,x}\left[U(X_{T}^{\alpha,{\scriptscriptstyle \Sigma}})+\lambda_{0}\int_{t}^{T}F(\Sigma_{s})ds\right]\right\} ,\label{eq:value function}
\end{equation}
because the uncertain drift case can be studied in a similar manner.

This problem is known as a \textit{static game}, and the function
$\underline{u}(t,x)$ is called the \textit{lower value of the static
game}. If we reverse the moving order of the two players, we obtain
the \textit{upper value of the static game}, which is 
\begin{equation}
\bar{u}(t,x)=\adjustlimits\inf_{\Sigma\in\mathcal{B}}\sup_{\alpha\in\mathcal{A}}\left\{ \mathbb{{E}}^{t,x}\left[U(X_{T}^{\alpha,{\scriptscriptstyle \Sigma}})+\lambda_{0}\int_{t}^{T}F(\Sigma_{s})ds\right]\right\} .\label{eq:u_bar}
\end{equation}

Note that $X_{s}^{\alpha,{\scriptscriptstyle \Sigma}},\forall s\in[t,T]$
denotes a process controlled by processes $\alpha,\Sigma$. When $X_{s}^{\alpha,{\scriptscriptstyle \Sigma}}$
starts from an initial condition $(t,x)$, we write the expectation
of $f(X_{s}^{\alpha,{\scriptscriptstyle \Sigma}})$ as $\mathbb{{E}}^{t,x}\left[f(X_{s}^{\alpha,{\scriptscriptstyle \Sigma}})\right]$.

\subsection{Assumptions}

In this section, we make the following assumptions which will hold
throughout the paper. \begin{assumption} \textcolor{black}{\label{assu:U}
The utility function $U:\mathbb{R}\rightarrow\mathbb{R}$ }is a continuous,
increasing and concave function such that 
\begin{equation}
\Bigl|U(x)-U(\bar{x})\Bigr|\leq Q(\left|x\right|,\left|\bar{x}\right|)\left|x-\bar{x}\right|,\label{eq:polyassumption}
\end{equation}
where $Q(\left|x\right|,\left|\bar{x}\right|)$ is a positive polynomial
function. \end{assumption} \begin{assumption} \label{assu:F convex}\textcolor{black}{T}he
penalty function {$F:B\rightarrow\mathbb{R}$} is a continuous convex
function, and $F$ attains its minimum in the interior of $B$. \end{assumption}

In addition to Definition \prettyref{def: An-admissible-vol}
and \prettyref{def: An-admissible-alpha}, we need the following conditions
to ensure the existence and uniqueness of a strong solution of the
SDE \eqref{eq:dynamic of X}.

\begin{assumption} \label{assu:finiteness}For any $\Sigma_{s,s\in[t,T]}\in B$
and $\alpha_{s,s\in[t,T]}\in A$, we have 
\[
\mathbb{E}\Bigl[\int_{t}^{T}\Bigl|F(\Sigma_{s})\Bigr|ds\Bigr]<\infty,
\]
and for any fixed value $x_{0}$, 
\[
\mathbb{E}\Bigl[\int_{t}^{T}\left|(\alpha_{s}^{\intercal}\mu+r-\alpha_{s}^{\intercal}\mathbf{r})x_{0}\right|^{2}+\left|{\color{black}{\color{black}{\color{brown}{\color{black}\alpha_{s}^{\intercal}\sigma_{s}x_{0}}}}}\right|^{2}ds\Bigr]<\infty.
\]
\end{assumption}

\textcolor{blue}{}%

\section{\textcolor{black}{Value functions of} two-player zero-sum stochastic
differential games\label{sec:The-classical-results}}

In order to complete the description of the game, we need to clarify
what information is available to the controllers at each time $s$.
For multi-stage discrete time games this can be formulated inductively.
However, this is problematic in continuous time, because control choices
can be changed instantaneously \citep[Chapter 11]{fleming2006controlled}.
To address this issue, \citet{fleming1989existence} adopted the idea
of a progressive strategy in a two-player zero-sum SDG, which is defined
as follows: 
\begin{defn}
An admissible strategy $\Gamma$ (resp. $\Delta$) for the investor
(resp. market) on $[t,T]$ is a mapping $\Gamma:\mathcal{B}\rightarrow\mathcal{A}$
(resp. $\Delta:\mathcal{A}\rightarrow\mathcal{B}$ ) such that, for
any $s\in[t,T]$ and $\Sigma,\tilde{\Sigma}\in\mathcal{B}$ (resp.
$\alpha,\tilde{\alpha}\in\mathcal{A}$), $\Sigma(u)=\tilde{\Sigma}(u)$
(resp. $\alpha(u)=\tilde{\alpha}(u)$) for all $u\in[t,s]$ implies
$\Gamma(\Sigma)(u)=\Gamma(\tilde{\Sigma})(u)$ (resp. $\Delta(\alpha)(u)=\Delta(\tilde{\alpha})(u)$)
for all $u\in[t,s]$. The set of all admissible strategies for the
investor (resp. market) on $[t,T]$ is denoted by $\mathcal{N}$ (resp.
$\mathcal{M}$). 
\end{defn}
In the two-player zero-sum SDG, one player is allowed to strategically
adapt his control according to the control of his opponent in a non-anticipative
fashion. This is in contrast to the static game, in which the player
must choose his control without any knowledge of the opponent's choice.
Then, we may define another set of value functions using these admissible
strategies: the \textit{upper value function} of the two-player zero-sum
SDG is defined by 
\begin{equation}
\bar{v}(t,x)=\adjustlimits\sup_{\Gamma\in\mathcal{\mathcal{N}}}\inf_{\Sigma\in\mathcal{B}}\Bigl\{\mathbb{E}^{t,x}\Bigl[\lambda_{0}\int_{t}^{T}F(\Sigma_{s})ds+U(X_{T}^{{\scriptscriptstyle \Gamma,\Sigma}})\Bigr]\Bigr\},\label{eq:upper}
\end{equation}
and the corresponding \textit{lower value function} is 
\begin{equation}
\underline{v}(t,x)=\adjustlimits\inf_{\Delta\in\mathcal{M}}\sup_{\alpha\in\mathcal{A}}\Bigl\{\mathbb{E}^{t,x}\Bigl[\lambda_{0}\int_{t}^{T}F(\Delta_{s})ds+U(X_{T}^{\alpha,{\scriptscriptstyle \Delta}})\Bigr]\Bigr\}.\label{eq:lower}
\end{equation}

The terms ``lower'' and ``upper'' are not obvious at first glance,
one might first guess the opposite because $\inf\sup\geq\sup\inf$.
We will justify $\underline{v}\leq\bar{v}$ in Corollary \prettyref{cor:u<v}
using the comparison principle.

\section{Existence of a value for the differential games\label{sec:Existence-of-saddle}}

In this section, we prove that the four value functions defined in
the previous sections all coincide, i.e., $\underline{u}(t,x)=\underline{v}(t,x)=\bar{v}(t,x)=\bar{u}(t,x)$.
This is established via the following propositions. 
\begin{prop}
\label{prop001} The four value functions defined in \prettyref{sec:Problem-Formulation}
and \prettyref{sec:The-classical-results} satisfy the following inequalities:
\begin{equation}
\underline{u}(t,x)\leq\underline{v}(t,x)\leq\bar{v}(t,x)\leq\bar{u}(t,x).\label{eq:value order}
\end{equation}
\end{prop}
\begin{proof}
The inequality $\underline{v}(t,x)\leq\bar{u}(t,x)$ holds because
$\mathcal{M}$ contains constant mappings, i.e., $\Delta(\alpha)=\Sigma$
for any $\alpha\in\mathcal{A}$ and fixed $\Sigma\in\mathcal{B}$.
Similarly, $\underline{u}(t,x)\leq\bar{v}(t,x)$ holds because $\mathcal{N}$
contains a copy of $\mathcal{A}$. Then for all $\alpha\in\mathcal{A}$
and $\epsilon>0$, there exists some $\bar{\Delta}$ such that 
\[
\inf_{\Delta\in\mathcal{M}}\sup_{\alpha\in\mathcal{A}}J\left(t,x,\alpha,\Delta(\alpha)\right)+\epsilon\geq\sup_{\alpha\in\mathcal{A}}J\left(t,x,\alpha,\bar{\Delta}(\alpha)\right)\geq J\left(t,x,\alpha,\bar{\Delta}(\alpha)\right)\geq\inf_{\Sigma\in\mathcal{B}}J(t,x,\alpha,\Sigma).
\]
So $\underline{u}(t,x)\leq\underline{v}(t,x)$. A similar argument
gives us $\bar{v}(t,x)\leq\bar{u}(t,x)$. Hence we have 
\[
\underline{u}\leq\underline{v}\leq\bar{u},\qquad\underline{u}\leq\bar{v}\leq\bar{u}.
\]
In order to complete the proof, it suffices to show that $\underline{v}(t,x)\leq\bar{v}(t,x)$.
This is proven in Corollary \prettyref{cor:u<v}. 
\end{proof}

\begin{prop}
\label{prop:concave in alpha}Let $U$ be a continuous, increasing
and concave utility function on $\mathbb{R}$, suppose that Assumption~\prettyref{assu:F convex}
holds, then $\underline{u}(t,x)=\underline{v}(t,x)=\bar{v}(t,x)=\bar{u}(t,x)$. 
\end{prop}
\begin{proof}
See Appendix \ref{sec:Appendix-proof in section 4}.
\end{proof}

Using Proposition \prettyref{prop:concave in alpha}, we can conclude
that there exists a value for the two-player zero-sum SDG, i.e., $\underline{v}=\bar{v}$.
We focus on the analysis of $\bar{v}(t,x)$ in the following sections.

\section{Dynamic programming principle \label{sec:The-Dynamic-Programming}}

If the drift and volatility functions of dynamics \eqref{eq:dynamic of X}
and the utility function $U$ were bounded and $U$ was Lipschitz
continuous, we could apply the results of \citet{fleming1989existence}
directly. However, in our model, the drift and volatility functions
are unbounded and $U$ is only locally Lipschitz continuous. So we
must extend the classical results and use localization techniques
to prove that the value function $\bar{v}(t,x)$ defined in \eqref{eq:upper}
satisfies the Dynamic Programming Principle (DPP). The DPP is widely
used in numerical methods, such as the least squares Monte Carlo method.

Before presenting the main result, we require the following important
property of the value function. 
\begin{prop}
\label{prop:continuous in x}Suppose that Assumptions \prettyref{assu:U}
and \prettyref{assu:finiteness} hold true. Then the value function
$\bar{v}(t,x)$ \eqref{eq:upper} is locally Lipchitz continuous w.r.t
$x$. There exists a positive polynomial function $\Phi$ such that
\begin{equation}
\Bigl|\bar{v}(t,x)-\bar{v}(t,\bar{x})\Bigr|\leq\Phi(\left|x\right|,\left|\bar{x}\right|)\left|x-\bar{x}\right|,\quad\forall(t,x)\in[0,T]\times\mathbb{R}.\label{eq:value function polynomial growth}
\end{equation}
\end{prop}

\begin{proof}
See Appendix \ref{sec:Appendix-proof in section 5}.
\end{proof}

We are now in the position to present a main result in this paper. 
\begin{thm}[Dynamic Programming Principle]
\label{thm:The-Dynamic-Programming}

Suppose that Assumptions \prettyref{assu:U}, \prettyref{assu:F convex}
and \prettyref{assu:finiteness} hold true. Define the value function
$\bar{v}(t,x)$ by \eqref{eq:upper} for $(t,x)\in[0,T]\times\mathbb{R}$. Let $t+\theta$ be a stopping time, 
then, for $t\leq t+\theta\leq T$, we have 
\begin{equation}
\bar{v}(t,x)=\adjustlimits\sup_{\Gamma\in\mathcal{N}}\inf_{\Sigma\in\mathcal{B}}\Bigl\{\mathbb{E}^{t,x}\Bigl[\lambda_{0}\int_{t}^{t+\theta}F(\Sigma_{s})ds+\bar{v}(t+\theta,X_{t+\theta}^{{\scriptscriptstyle \Gamma,\Sigma}})\Bigr]\Bigr\}.\label{eq:DPP1}
\end{equation}
\end{thm}

\begin{proof}
See Appendix \ref{sec:Appendix-proof of DPP}.
\end{proof}

As a consequence of the DPP, the value function $\bar{v}(t,x)$ satisfies
the following property. 
\begin{cor}
\label{cor:continuous in time}

Suppose that Assumptions \prettyref{assu:U}, \prettyref{assu:F convex}
and \prettyref{assu:finiteness} hold true. Then the value function
$\bar{v}(t,x)$ defined in \eqref{eq:upper} is H\"older continuous
in $t$ on $[0,T]$. 
\end{cor}

\begin{proof}
See Appendix \ref{sec:Appendix-proof of continuous in time }.
\end{proof}

\section{Viscosity solution of the HJBI equation\label{sec:Viscosity-solution-HJBI} }

In this section, we prove that the value function is the unique viscosity
solution of a Hamilton-Jacobi-Bellman-Isaacs equation. In \prettyref{subsec:Existence-of-viscosity},
we prove the existence of the viscosity solution, and we state the
uniqueness of this viscosity solution in \prettyref{subsec:The-Uniqueness-of-solution}.

\subsection{Existence of a viscosity solution of the HJBI Equation\label{subsec:Existence-of-viscosity}}

Now we state another main result in this paper; the proof is a modification
of \citet{talay2002worst}. 
\begin{thm}
\label{thm:existence-of-viscosity}Suppose that Assumptions \prettyref{assu:U},
\prettyref{assu:F convex} and \prettyref{assu:finiteness} hold true.
Then the value function $\bar{v}(t,x)$ defined in \eqref{eq:upper}
is a viscosity solution of the HJBI equation

\begin{align}
\begin{cases}
\frac{\partial v}{\partial t}(t,x)+H(t,x,\frac{\partial v}{\partial x}(t,x),\frac{\partial^{2}v}{\partial x^{2}}(t,x))=0 & \text{in }\,[0,T)\times\mathbb{R}\\
v(T,x)=U(x) & \text{on }\,[T]\times\mathbb{R},
\end{cases}\label{eq:HJBI1}
\end{align}
where 
\begin{equation}
H(t,x,p,M)=\adjustlimits\inf_{\mathbf{\Sigma}\in B}\sup_{\mathbf{a}\in A}
\left\{ \lambda_{0}F(\mathbf{\Sigma})+(\mathbf{a}^{\intercal} (\mu - \mathbf{r}) +r)xp+\frac{1}{2}tr\left(\mathbf{a}^{\intercal}\mathbf{\Sigma}\mathbf{a}x^{2}M\right)\right\} ,\label{eq:Hamiltonian1}
\end{equation}
for $(t,x,p,M)\in[0,T]\times\mathbb{R}\times\mathbb{R}\times\mathbb{R}$. 
\end{thm}

\begin{proof}
See Appendix \ref{sec:Appendix-proof of existence of viscosity solution}.
\end{proof}

\subsection{Comparison principle for the HJBI Equation\label{subsec:The-Uniqueness-of-solution}}

In this subsection, we present the comparison principle for equation
\eqref{eq:HJBI1}, which implies the uniqueness of the viscosity solution
of the HJBI equation. We can adapt the proof from \citet[Theorem 4.4.4]{pham2009continuous}
for an HJB equation and straightforwardly extend it to HJBI equations with two controls. 
\begin{thm}
\label{thm:Comparison-Principle-beta}Comparison Principle (\citealt[Theorem 4.4.4]{pham2009continuous}).

Let Assumptions \prettyref{assu:U}, \prettyref{assu:F convex} and
\prettyref{assu:finiteness} hold true. Define the HJBI equation as
\begin{multline}
-\frac{\partial v}{\partial t}(t,x)-\adjustlimits\inf_{\mathbf{\Sigma}\in B}\sup_{\mathbf{a}\in A}
\left\{ \lambda_{0}F(\mathbf{\Sigma})+(\mathbf{a}^{\intercal} (\mu - \mathbf{r}) +r )x\frac{\partial v}{\partial x}(t,x)+\frac{1}{2}tr\left(\mathbf{a}^{\intercal}\mathbf{\Sigma}\mathbf{a}x^{2}\frac{\partial^{2}v}{\partial x^{2}}(t,x)\right)\right\} =0,\\
\text{ }\text{for}\,(t,x)\in[0,T)\times\mathbb{R}.\label{eq:comparePDE-1}
\end{multline}
Let $U$ (resp.\! $V$) be a u.s.c.\! viscosity subsolution (resp.\!
l.s.c.\! supersolution) with polynomial growth condition to equation
\eqref{eq:comparePDE-1}. If $U(T,\cdot)\leq V(T,\cdot)$ on $\mathbb{R}$,
then $U\leq V$ on $[0,T]\times\mathbb{R}$. 
\end{thm}
As a consequence of the comparison principle, the function $\bar{v}(t,x)$
\eqref{eq:upper} is in fact the unique viscosity solution of the
HJBI equation \eqref{eq:HJBI1}. 
\begin{cor}
\label{cor:u<v}Let Assumptions \prettyref{assu:U}, \prettyref{assu:F convex}
and \prettyref{assu:finiteness} hold true. Define the lower and upper
value functions of the two-player zero-sum SDG by \eqref{eq:lower}
and \eqref{eq:upper}. Then 
\[
\underline{v}(t,x)\leq\bar{v}(t,x)\qquad\text{for}\,(t,x)\in[0,T]\times\mathbb{R}.
\]
\end{cor}
\begin{proof}
From \prettyref{thm:existence-of-viscosity}, $\bar{v}(t,x)$ is a
viscosity solution of the HJBI equation \eqref{eq:HJBI1}. Let $\phi\in C^{\infty}([0,T)\times\mathbb{R})$
be a test function such that $(t_{0},x_{0})\in[0,T)\times\mathbb{R}$
is a local minimum of $\bar{v}-\phi$. Using the viscosity supersolution
property of $\bar{v}(t,x)$, we have 
\[
-\frac{\partial\phi}{\partial t}(t_{0},x_{0})-H(t_{0},x_{0},\frac{\partial\phi}{\partial x}(t_{0},x_{0}),\frac{\partial^{2}\phi}{\partial x^{2}}(t_{0},x_{0}))\geq0,
\]
where $H(t,x,p,M)$ is defined by \eqref{eq:Hamiltonian1}. Define
\begin{equation}
\tilde{H}(t,x,p,M)=\adjustlimits\sup_{\mathbf{a}\in A}\inf_{\mathbf{\Sigma}\in B}\left\{ \lambda_{0}F(\mathbf{\Sigma})
+(\mathbf{a}^{\intercal} (\mu - \mathbf{r}) +r)xp+\frac{1}{2}tr\left(\mathbf{a}^{\intercal}\mathbf{\Sigma}\mathbf{a}x^{2}M\right)\right\} .\label{eq:Hamiltonian1-1}
\end{equation}
It is obvious that $H\geq\tilde{H}$, so 
\[
-\frac{\partial\phi}{\partial t}(t_{0},x_{0})-\tilde{H}(t_{0},x_{0},\frac{\partial\phi}{\partial x}(t_{0},x_{0}),\frac{\partial^{2}\phi}{\partial x^{2}}(t_{0},x_{0}))\geq0\quad\text{in }\,[0,T)\times\mathbb{R}.
\]

Thus $\bar{v}(t,x)$ is a supersolution of the HJBI equation 
\[
\frac{\partial v}{\partial t}(t,x)+\tilde{H}(t,x,\frac{\partial v}{\partial x}(t,x),\frac{\partial^{2}v}{\partial x^{2}}(t,x))=0,\,(t,x)\in[0,T)\times\mathbb{R}.
\]

Using the results of \citet{fleming1989existence} and a similar argument,
we can prove the lower value function $\underline{v}(t,x)$ \eqref{eq:lower}
is the unique viscosity solution of the HJBI equation 
\begin{align}
\begin{cases}
\frac{\partial v}{\partial t}(t,x)+\tilde{H}(t,x,\frac{\partial v}{\partial x}(t,x),\frac{\partial^{2}v}{\partial x^{2}}(t,x))=0 & \text{in }\,[0,T)\times\mathbb{R}\\
v(T,x)=U(x) & \text{on }\,[T]\times\mathbb{R}.
\end{cases}\label{eq:HJBI1-1}
\end{align}
Finally, by the comparison principle, we have $\underline{v}(t,x)\leq\bar{v}(t,x)$,
as required. 
\end{proof}

\section{Numerical results \label{sec:Numerical-results}}

In this section, we provide a few numerical examples
with commonly used utility functions. We first establish an analytical
solution in the case of the Logarithmic utility function. Then we
numerically approximate the value functions for both the Logarithmic
and CRRA utility functions using an implicit finite difference method, a control randomization method, and a Generative Adversarial Network method.

\subsection{Analytical solution\label{subsec:Analytical-Solution}}

In the first example, we consider $U(x)=\ln(x)$ and the penalty function
$F(\sigma_{t}^{2})=(\sigma_{t}-\sigma_{0})^{2}$. It is possible to
find the explicit solution for the value function as well as the optimal
controls. Writing $X_{T}$ explicitly, the value function becomes:
\begin{equation}
\begin{aligned}\bar{v}(t,x) & =\adjustlimits\sup_{\alpha\in\mathcal{A}}\inf_{\sigma^{2}\in\mathcal{B}}\left\{ \mathbb{E}^{t,x}\Bigl[\ln(x)+\int_{t}^{T}(\alpha_{s}\mu+(1-\alpha_{s})r-\frac{1}{2}\alpha_{s}^{2}\sigma_{s}^{2})ds+\int_{t}^{T}\alpha_{s}\sigma_{s}dW_{s}+\lambda_{0}\int_{t}^{T}(\sigma_{s}-\sigma_{0})^{2}ds\Bigr]\right\} \\
 & =\adjustlimits\sup_{\alpha\in\mathcal{A}}\inf_{\sigma^{2}\in\mathcal{B}}\left\{ \mathbb{E}^{t,x}\Bigl[\ln(x)+\int_{t}^{T}\alpha_{s}\mu+(1-\alpha_{s})r-\frac{1}{2}\alpha_{s}^{2}\sigma_{s}^{2}+\lambda_{0}(\sigma_{s}-\sigma_{0})^{2}ds\Bigr]\right\} .
\end{aligned}
\label{eq:value function explicit}
\end{equation}
To find the optimal $\alpha_{s}$ and $\sigma_{s}^{2}$, we can differentiate
instantaneously the integrand $\alpha_{s}(\mu-r)+r-\frac{1}{2}\alpha_{s}^{2}\sigma_{s}^{2}+\lambda_{0}(\sigma_{s}-\sigma_{0})^{2}$
with respect to $\alpha_{s}$ and $\sigma_{s}^{2}$ respectively.
Then we obtain the following optimality conditions: 
\begin{align}
\hat{\alpha}_{s} & =\frac{\mu-r}{\hat{\sigma}_{s}^{2}},\label{eq:simultaneous1}\\
-\frac{1}{2}\hat{\alpha}_{s}^{2}+\lambda_{0}(1-\frac{\sigma_{0}}{\hat{\sigma}_{s}}) & =0,\label{eq:simultaneous2}
\end{align}
which leads to a quartic equation 
\begin{align}
0 & =\hat{\sigma}_{s}^{4}-\sigma_{0}\hat{\sigma}_{s}^{3}-\dfrac{(\mu-r)^{2}}{2\lambda_{0}}.\label{eq:quartic}
\end{align}

The optimal $\hat{\sigma}_{s}$ and $\hat{\alpha}_{s}$ can be solved
from equation \eqref{eq:quartic} explicitly; we provide the solution
in the Appendix \ref{sec:Appendix-explicit solution}. The equation  \eqref{eq:quartic} always has a real positive root, hence
the optimal volatility $\hat{\sigma}_{s}\in B$ and optimal strategy
$\hat{\alpha}_{s}\in A$.  By substituting the optimal controls into
\eqref{eq:value function explicit}, we obtain the analytical solution
of the value function. From equations \eqref{eq:simultaneous1}--\eqref{eq:simultaneous2},
we observe that the optimal volatility and investment strategy are
both constants, being independent of the wealth $X_{s}$ and the time
$s$. The classical optimal portfolio strategy given by Merton is
also a constant, where $\alpha^{*}=\frac{\mu-r}{\sigma^{2}(1-\gamma)}$
for CRRA utility functions. However, in our problem, it is not possible to
find an analytical solution for a power utility function. We will
use numerical methods to estimate the values in the next subsection.
It is worth mentioning that, when $U(x)=\ln(x)$,  we can apply the above method to portfolios with multiple risky assets and get the analytical solutions by solving a system of optimality conditions.  The detailed process is very similar, hence omitted here. 
Moreover, the reference volatility $\sigma_{0}$ is not necessarily a constant, it
can be a local volatility depending on time and stock price.
However, for multiple assets, it would increase the dimension of the problem.

\subsection{Comparison of robust and non-robust portfolios with  Monte Carlo simulation \label{subsec:Monte Carlo comparison}}

In this section, we implement our robust strategy using Monte Carlo simulations, and compare the performance of robust and non-robust portfolios. 

As we know, in the real world volatility estimates are noisy and biased, though likely to oscillate around a reference value in the long run. In the first experiment, 
we have a reference covariance matrix $\Sigma_0$, which is estimated according to historical data. 
We assume that the real-world covariance is \textit{the reference covariance  $\Sigma_0 $ plus some noise}.
We construct robust and non-robust portfolios consisting of two risky assets and one risk-free asset. 
For the robust portfolio, we use $\lambda_0 F(\Sigma_s) = \lambda_0 \bigl\Vert  \Sigma_s - \Sigma_0 \bigl\Vert_2^2$ ( $\bigl\Vert \cdot \bigl\Vert_2$ denotes the usual Frobenius norm) as the penalty function,  then the analytical robust investment strategy $(\hat{\alpha}_s^1, \hat{\alpha}_s^2)$ can be calculated in a similar method to the one in section \ref{subsec:Analytical-Solution}. 
For the non-robust one, we use $\Sigma_0$ as the covariance, then 
calculate the  non-robust strategy $(\alpha_s^1, \alpha_s^2)$ accordingly. 
Assuming the real covariance matrix during the investment process is $\Sigma _{\mathrm{real}} = \Sigma_0 + \varepsilon \times \text{noise}$,  where the noise follows a standard normal distribution $\mathcal{N} (0,1)$ and $\varepsilon $ is the magnitude of the noise, 
we use Monte Carlo simulations  to estimate the expected utility function  
\begin{equation}
\mathbb{E} \Bigl[ \ln(X_T) \Bigr] = \mathbb{E}^{t,x}\Bigl[\ln(x)+\int_{t}^{T}\alpha_{s}^{\intercal} (\mu - \mathbf{r}) + r -\frac{1}{2} \alpha_{s}^{\intercal} \Sigma _{\mathrm{real}} \alpha_{s} ds\Bigr].
\label{robust and nonrobust}
\end{equation}
We substitute $ \alpha_s = (\hat{\alpha}_s^1, \hat{\alpha}_s^2)$ in \eqref{robust and nonrobust} for the robust portfolio, and $\alpha_s = (\alpha_s^1, \alpha_s^2)$ for the non-robust one. 

The results with various $\lambda_0$ are shown in Figures \ref{fig:robsut_simulation_la001} to \ref{fig:robsut_simulation_la70}, where we used  $2 \times 10^5$ paths in the simulation and the initial wealth $X_0 = 1$. We can observe that the robust portfolio may underperform when there is little noise. But, as the noise size $\varepsilon$ increases, the robust strategy will outperform the non-robust strategy eventually. 
Comparing Figures \ref{fig:robsut_simulation_la001}, \ref{fig:robust_simulation_la1} and  \ref{fig:robsut_simulation_la70}, we can find that when the penalty is relatively weak ($\lambda_0 = 0.01$), it takes a bigger noise size for the robust strategy to outperform. When the penalty is stiff ($\lambda_0 = 70$), the robust strategy will outperform with a very small noise size. The robust expected utility is almost a constant for all sizes of noise in Figure \ref{fig:robsut_simulation_la001}, meaning that our model is very robust to changes in market circumstances. Among the three values of $\lambda_0$ illustrated, Figure \ref{fig:robsut_simulation_la70} is probably  the most attractive to investors. When the reference $\Sigma_0$ is perfect, the robust portfolio only loses to the non-robust one by a little, but when $\Sigma_0$ is wrong, the robust portfolio outperforms the non-robust one by a large amount. It means the price we pay for the robustness is tolerable, but the potential reward is substantial.

Define the \textit{crossing point $\varepsilon$ } as the value of $\varepsilon$ for which the robust expected utility matches the non-robust expected utility.
Figure \ref{fig:epsilon_vs_lambda} depicts how the crossing point $\varepsilon$ varies with respect to $\lambda_0$. 
It tells us how much should our reference covariance be wrong for the robust portfolio to outperform the non-robust portfolio. 
The behaviour of the robust portfolio varies with $\lambda_0$. For a certain $\varepsilon$, by looping over a range of $\lambda_0$, we can find the one giving us the maximal  robust expected utility. This relation is plotted in Figure \ref{fig:best lambda}.
With this plot, if we know how confident we are with the reference $\Sigma_0$ (i.e., the value of $\varepsilon$), we can choose  the best $\lambda_0$ for robust portfolio allocation. 

\begin{figure}[h]
\centering
\begin{minipage}{.5\textwidth}
  \centering
  \includegraphics[scale=0.4]{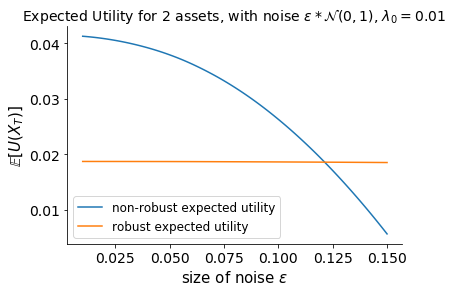} 
  \captionof{figure}{Expected utility with $\lambda_0 = 0.01$}
  \label{fig:robsut_simulation_la001}
\end{minipage}%
\begin{minipage}{.5\textwidth}
  \centering
  \includegraphics[scale=0.4]{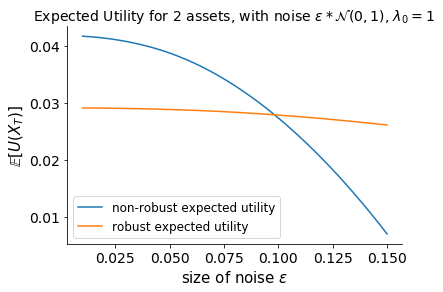}
  \captionof{figure}{Expected utility with $\lambda_0 = 1$}
  \label{fig:robust_simulation_la1}
\end{minipage}
\begin{minipage}{.5\textwidth}
  \centering
  \includegraphics[scale=0.4]{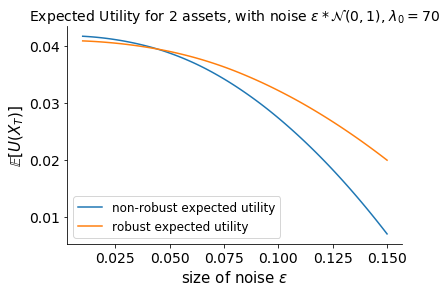} 
  \captionof{figure}{Expected utility with $\lambda_0 = 70$}
  \label{fig:robsut_simulation_la70}
\end{minipage}%
\end{figure}

\begin{figure}[h]
\centering
\begin{minipage}{.5\textwidth}
  \centering
  \includegraphics[scale=0.4]{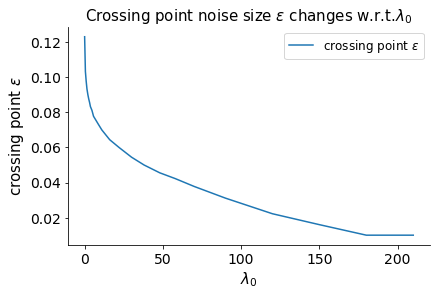} 
  \captionof{figure}{}
  \label{fig:epsilon_vs_lambda}
\end{minipage}%
\begin{minipage}{.5\textwidth}
  \centering
  \includegraphics[scale=0.4]{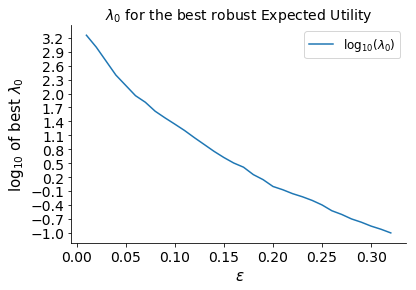}
  \captionof{figure}{}
  \label{fig:best lambda}
\end{minipage}
\end{figure}

\subsection{Comparison of robust and non-robust portfolios with  empirical market data \label{subsec:Empirical Experiment}}

In the second experiment, we implement the robust and non-robust strategies with empirical market data.
We have  $1007 $ portfolios, and  we construct each portfolio according to robust and non-robust allocations, respectively.
Each portfolio consists of $2$ risky assets and $1 $ risk-free asset, with a maturity of $T=1$ year. The portfolios' starting dates range from 02/04/15 to 03/04/19 (for example, the $1$st portfolio starts on 02/04/15 and lasts for one year, the $1007$th portfolio starts on 03/04/19 and lasts for one year as well). 
We choose the S$\&$P500 (\textsuperscript{$\wedge$}GSPC) and SPDR Gold Shares (GLD)\footnote{Stock prices are downloaded from Yahoo Finance.} as our risky assets and use a constant interest rate $ r = 0.015 $.
For a specific portfolio, we set $\Sigma_0$ to be the sample covariance estimator of the $5$ years of daily relative returns before the starting date. The estimated annual expected  returns $\mu_1, \mu_2$ are the exponentially weighted moving average of the  daily relative returns with a $5$-year lookback window and $2.75$-year half-life. With a decay parameter $ \beta = 0.999$, for the $n$th portfolio,
$\mu_{i,i=1,2} = 252 \times \frac{1}{1-\beta^{1260}}\sum_{t=0}^{1260} (1-\beta)\beta^t \frac{S^i_{n-t} - S^i_{n-t-1}}{S^i_{n-t-1}}$.

In this experiment, we use a logarithmic utility function and a penalty function $\lambda_0 F(\Sigma_s) = \lambda_0 \bigl\Vert  \Sigma_s - \Sigma_0 \bigl\Vert_2^2$. At the beginning of the investment process for each portfolio, 
we estimate parameters $\mu_1, \mu_2, \Sigma_0$ and then
compute the robust and non-robust portfolio allocations accordingly. Starting from an initial wealth $ X_{0}  =1 $, the wealth of  the non-robust portfolio evolves as
\begin{multline}
X_{n+1}  =  X_{n} \exp \Bigl\{ \alpha_n^1 \frac{S^1_{n+1}-S^1_{n}}{S^1_n} +  \alpha_n^2 \frac{S^2_{n+1}-S^2_{n}}{S^2_n} + (1- \alpha_n^1 - \alpha_n^2) r \Delta t \\
 -\frac{1}{2} \left[ \alpha_n^1 \left(\frac{S^1_{n+1}-S^1_{n}}{S^1_n} - \mu_1 \Delta t\right) + \alpha_n^2 \left(\frac{S^2_{n+1}-S^2_{n}}{S^2_n} - \mu_2 \Delta t\right) \right]^2 \Bigr\} \,,\,n\in[0,251],
 \label{expirical_dynamics}
 \end{multline}
where $(\alpha_n^1, \alpha_n^2)$ are the non-robust allocations on day $n$. For the wealth of the robust portfolio, just replace $(\alpha_n^1, \alpha_n^2)$ with the robust allocations $(\hat{\alpha}_n^1, \hat{\alpha}_n^2)$ in \eqref{expirical_dynamics}. Finally, by averaging the $\ln(X_T)$ of all the portfolios, we get the expected utility function. 

Figures \ref{fig:la001}--\ref{fig:la1000} present the terminal wealth $X_T$ of the $1007$  robust and non-robust portfolios. For a small $\lambda_0$, the robust portfolios are very stable. No matter how the market changes, the robust terminal wealth stays around $1$. As $\lambda_0$ increases, the robust portfolios start to show fluctuations. Eventually, their behaviour converges to that of the non-robust portfolios as $\lambda_0$ approaches to infinity, which corresponds to the non robust case.
This behaviour is consistent with our expectations. The penalty function is not playing its role when $\lambda_0$ is close to zero. Hence the robust allocations are optimal for the most chaotic market situations, and the investment strategies are very conservative. 
As $\lambda_0$ becomes larger, the penalty function comes into play and prevents extreme volatilities. As a consequence, the robust strategies are less conservative, and portfolios will show more fluctuations under regime changes. 

We show the robust and non-robust expected utilities in Figure \ref{fig:10year_v_la0}. It depicts how $\mathbb{E}[\ln (X^{\alpha^1, \alpha^2}_T)]$ and $\mathbb{E}[\ln (X^{\hat{\alpha}^1, \hat{\alpha}^2}_T)]$ change w.r.t. $\lambda_0$. We can compare this plot with Figures \ref{fig:robsut_simulation_la001}, \ref{fig:robust_simulation_la1}, \ref{fig:robsut_simulation_la70} and \ref{fig:best lambda} in section \prettyref{subsec:Monte Carlo comparison}. For a given amount of noise, the robust portfolio may underperform for small $\lambda_0$, but the value will increase gradually and reach a highest point. Finally, the robust  expected utility will converge to the non-robust one. 

\begin{figure}[h]
\centering
\begin{minipage}{.5\textwidth}
  \centering
  \includegraphics[scale=0.4]{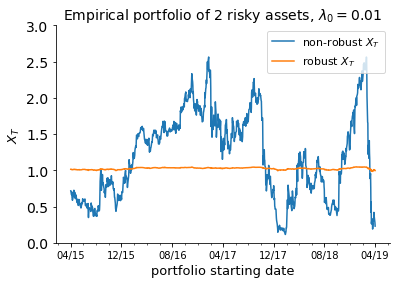} 
  \captionof{figure}{Terminal wealth with $\lambda_0=0.01$}
  \label{fig:la001}
  \vspace{2ex}
\end{minipage}%
\begin{minipage}{.5\textwidth}
  \centering
  \includegraphics[scale=0.4]{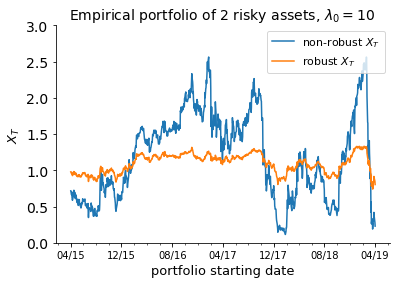}
  \captionof{figure}{Terminal wealth with $\lambda_0=10$}
   \label{fig:la10}
  \vspace{2ex}
\end{minipage}
\begin{minipage}{.5\textwidth}
  \centering
  \includegraphics[scale=0.4]{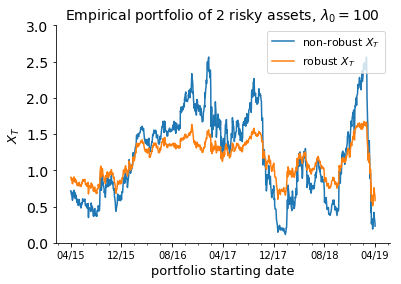} 
  \captionof{figure}{Terminal wealth with $\lambda_0=100$}
   \label{fig:la100}
\end{minipage}%
\begin{minipage}{.5\textwidth}
  \centering
  \includegraphics[scale=0.4]{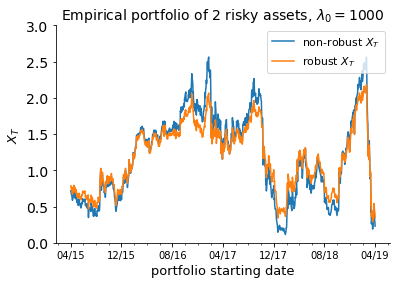} 
  \captionof{figure}{Terminal wealth with $\lambda_0=1000$}
  \label{fig:la1000}
\end{minipage}%
\end{figure}

\begin{figure}[h]
\begin{centering}
\includegraphics[scale=0.4]{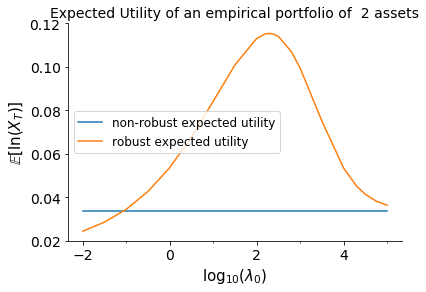}
\par\end{centering}
\caption{Empirical expected utility w.r.t. $\lambda_0$ 
\label{fig:10year_v_la0}}
\end{figure}

To illustrate the time evolution of the portfolio wealth,  we show the stock prices and  wealth of two portfolios,  starting on 2017-01-03 (Figure \ref{fig:portfolio_17_01_03}) and 2018-01-26 (Figure \ref{fig:portfolio_18_01_26}), respectively.
For the portfolio in Figure \ref{fig:portfolio_17_01_03}, the optimal non-robust allocations are $\alpha^1 = 5.778, \alpha^2 = -2.174$, and the robust allocations with $\lambda_0=200$ are $\hat{\alpha}^1 = 3.083, \hat{\alpha}^2 = -1.452 $. The allocations are both constant, independent of time. The S$\&$P500 keeps rising in Figure \ref{fig:stock_price_17_01_03}, while there are some fluctuations in the Gold price. Over the same period, the absolute performance of the non-robust portfolio is better all the way (Figure \ref{fig:portfolio_wealth_17_01_03}). 
For the portfolio in Figure \ref{fig:portfolio_18_01_26}, we have $\alpha^1 = 9.418, \alpha^2 = 0.301$, and 
$\hat{\alpha}^1 = 3.940, \hat{\alpha}^2 = -0.054 $. Since the proportions invested in Gold are small for both robust and non-robust portfolios, the trend of wealth is dominated by the price of S$\&$P500. 
There are two big drops happening in Feb. 2018 and Dec. 2018, respectively. These are also reflected in the portfolio wealth in Figure \ref{fig:portfolio_wealth_18_01_26}. 
However, compared with the non-robust strategy, the robust strategy is more conservative.  Hence, the robust portfolio loses less during the market shocks and outperforms the non-robust one.

\begin{figure}[h]
    \centering
    \subfloat[\label{fig:stock_price_17_01_03}]{{\includegraphics[scale=0.45]{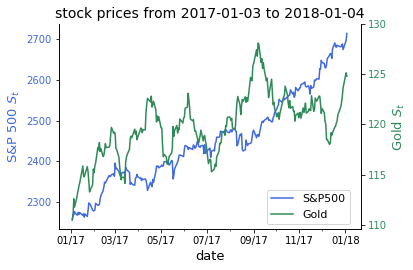} }}%
    \qquad \quad
    \subfloat[\label{fig:portfolio_wealth_17_01_03}]{{\includegraphics[scale=0.45]{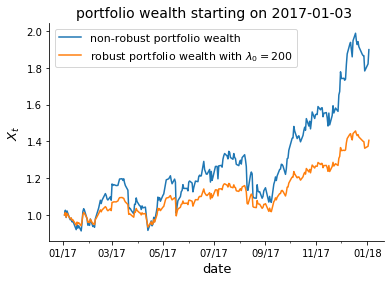} }}%
    \caption{The portfolio starting on 2017-01-03}%
    \label{fig:portfolio_17_01_03}%
\end{figure}

\begin{figure}[h]
    \centering
    \subfloat[\label{fig:stock_price_18_01_26}]{{\includegraphics[scale=0.45]{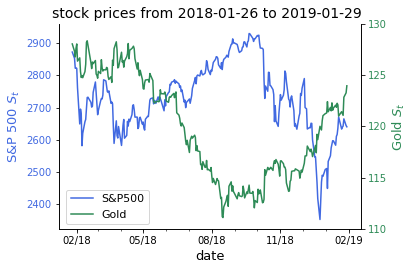} }}%
    \qquad \quad
    \subfloat[\label{fig:portfolio_wealth_18_01_26}]{{\includegraphics[scale=0.45]{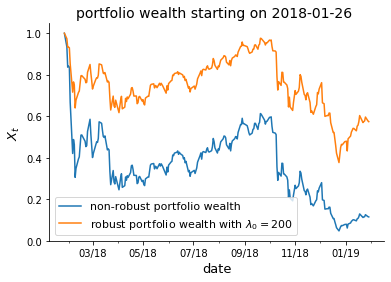} }}%
    \caption{The portfolio starting on 2018-01-26}%
    \label{fig:portfolio_18_01_26}%
\end{figure}

From the above empirical experiments and the Monte Carlo simulations from subsection \prettyref{subsec:Monte Carlo comparison} , we can see that, by adding this robust mechanism with a properly chosen $\lambda_0$,  the portfolio value can overcome a wrong covariance matrix estimate and is less vulnerable to sudden market shocks. 
Furthermore, unlike other robust methods which only consider the worst case, our model is more flexible and provides a greater range of more practical in-between option. 

\subsection{Implicit finite difference method\label{subsec:Finite-Difference-Method}}

In this section, we are computing the value function via an implicit finite difference method. 
We use the penalty function $\lambda_0 F(\sigma_{t}^{2})=\lambda_{0}(\sigma_{t}^{2})^{2}$ for simplicity. Then the HJBI equation is 
\begin{equation}
\bar{v}_{t}+H(t,x,\bar{v}_{x},\bar{v}_{xx})=0,\label{eq:numer_HJBI}
\end{equation}
where the Hamiltonian is defined by 
\begin{equation}
H(t,x,\bar{v}_{x},\bar{v}_{xx})=\inf_{\mathbf{\sigma}^{2}}\sup_{\mathbf{a}}\Bigl\{\mathbf{a}(\mu-r)x\bar{v}_{x}+rx\bar{v}_{x}+\frac{1}{2}\mathbf{a}^{2}\sigma^{2}x^{2}\bar{v}_{xx}+\lambda_{0}(\sigma^{2})^{2}\Bigr\}.\label{eq:Hamiltonian in PDE}
\end{equation}
Solving for the optimal controls in \eqref{eq:Hamiltonian in PDE}
using the first order condition, we obtain $\hat{\mathbf{a}}=-\frac{(\mu-r)x\bar{v}_{x}}{\sigma^{2}x^{2}\bar{v}_{xx}}$
and $\hat{\sigma}^{2}=\Bigl(-\frac{(\mu-r)^{2}\bar{v}_{x}^{2}}{4\lambda_{0}\bar{v}_{xx}}\Bigr)^{1/3}$.
Substituting $\hat{\mathbf{a}}$ and $\hat{\sigma}^{2}$ into the
PDE \eqref{eq:numer_HJBI}, we obtain 
\begin{alignat*}{1}
\bar{v}_{t}+C\bar{v}_{x}^{\frac{4}{3}}(-\bar{v}_{xx})^{-\frac{2}{3}}+rx\bar{v}_{x} & =0,
\end{alignat*}
where $C=(3\times2^{-\frac{4}{3}})\lambda_{0}^{\frac{1}{3}}(\mu-r)^{\frac{4}{3}}$.
Note we have shown in \prettyref{sec:Existence-of-saddle} that $\bar{v}_{xx}<0$.

Since the PDE \eqref{eq:numer_HJBI} is non-linear, in order to use
the implicit finite difference method, we first linearize the function
$H$ with respect to the second order term via the Legendre transform.
This method was also used by \citet{jonsson2002partial,jonsson2002optimal}
to solve nonlinear HJB equations. We also combine the linearization
step with a fixed-point iteration scheme.

Define $H^{*}$ as the Legendre transform of $H$ with respect to
the second order term; it is given by 
\begin{alignat*}{1}
H^{*}(a) & =-C_{2}a^{\frac{2}{5}}\bar{v}_{x}^{\frac{4}{5}}-rx\bar{v}_{x},
\end{alignat*}
where $C_{2}=\frac{5}{3}(\frac{2}{3})^{-\frac{2}{5}}C^{\frac{3}{5}}$.
Hence, we can represent $H(\bar{v}_{xx})$ as the supremum of linear
functions of $\bar{v}_{xx}$, 
\begin{alignat}{1}
H(\bar{v}_{xx}) & =\sup_{a}\biggl\{ a\cdot\bar{v}_{xx}-H^{*}(a)\biggr\}.\label{eq:numer_linearized H}
\end{alignat}
It is difficult to check the condition for stability in our PDE as
the optimal $a$ is unknown. Fortunately, implicit finite difference
methods have a weaker requirement for stability than explicit finite
difference methods.

We set the time grid as $0,1,...,n,n+1,...,N$, and the spatial grid
as $1,2,...i,i+1,...M$. With the maturity $T=1$, we use a constant
time step $\Delta t=\frac{T}{N}$ and a constant spatial step $\Delta x$.
We apply a forward approximation for $\bar{v}_{t}$, a central approximation
for $\bar{v}_{x}$, and a standard approximation for $\bar{v}_{xx}$.
Working backward in the implicit scheme, at each time step $n$, the
optimal $\hat{a}$ in \eqref{eq:numer_linearized H} is the solution
of the first order condition $\bar{v}_{xx}^{n}+C_{2}(\bar{v}_{x}^{n})^{\frac{4}{5}}\frac{2}{5}\hat{a}^{-\frac{3}{5}}=0,$
or equivalently, 
\begin{equation}
\hat{a}=\frac{2}{3}C(\bar{v}_{x}^{n})^{\frac{4}{3}}(-\bar{v}_{xx}^{n})^{-\frac{5}{3}}\eqqcolon f(\hat{a}).\label{eq:fixed point function}
\end{equation}
Although we do not have the true values for $\bar{v}^{n}$ as the values
of $\bar{v}^{n}$ depend on $\hat{a}$, we can use a fixed-point iteration
scheme to find the solution of equation \eqref{eq:fixed point function}.
First we make an initial guess $\hat{a}_{0}$ using the known values
$\bar{v}^{n+1}$, then iteratively generate a sequence $\hat{a}_{k,k=1,2,...}$
with $\hat{a}_{k}=f(\hat{a}_{k-1})$ until $\hat{a}_{k}$ converges.

Finally we can substitute the discrete approximations of the derivatives
into the HJBI equation \eqref{eq:numer_HJBI}, and we obtain the implicit
form: 
\begin{multline}
\left(\frac{\hat{a}(i)\Delta t}{\Delta x^{2}}-\frac{r\left(i\Delta x+x_{0}\right)\Delta t}{2\Delta x}\right)\bar{v}_{i-1}^{n}+\left(-1-\frac{2\hat{a}(i)\Delta t}{\Delta x^{2}}\right)\bar{v}_{i}^{n}+\left(\frac{\hat{a}(i)\Delta t}{\Delta x^{2}}+\frac{r\left(i\Delta x+x_{0}\right)\Delta t}{2\Delta x}\right)\bar{v}_{i+1}^{n}\\
=-\bar{v}_{i}^{n+1}-C_{2}\hat{a}(i)^{\frac{2}{5}}\Bigl(\frac{\bar{v}_{i+1}^{n+1}-\bar{v}_{i-1}^{n+1}}{2\Delta x}\Bigr)^{\frac{4}{5}}\Delta t.\label{eq:discretized implicit PDE}
 \end{multline}

Let $\mathbf{B}$ be the coefficient matrix, $K^{n}$ the value vector
at time $n$ and $F^{n+1}$ the right hand side of \eqref{eq:discretized implicit PDE}.
Then equation \eqref{eq:discretized implicit PDE} can be written
in a matrix notation: 
\[
\mathbf{B}K^{n}+G^{n}=F^{n+1},\qquad n=N-1,...,1,0.
\]
The algorithm for this method is summarized in \textbf{Algorithm \ref{alg:Implicit-Finite-Difference}}.

\begin{algorithm}[H]
\For{step $n=N : 1$}{
\begin{enumerate}
\item Solve $\mathbf{B}K^{n-1}+G^{n-1}=F^{n}$ using $\hat{a}_{0}(i)=g(\bar{v}_{i+1}^{n},\bar{v}_{i}^{n},\bar{v}_{i-1}^{n})$,
and get the value vector $K_{0}^{n-1}$ 
\item \label{enu:Solve value }Solve $\mathbf{B}K^{n-1}+G^{n-1}=F^{n}$
using $\hat{a}_{1}(i)=g(\bar{v}_{i+1}^{n-1},\bar{v}_{i}^{n-1},\bar{v}_{i-1}^{n-1})$,
where the values $\bar{v}^{n-1}$ are from $K_{0}^{n-1}$. Then get
the value vector $K_{1}^{n-1}$. 
\item Repeat step \ref{enu:Solve value } until $\bigl\Vert\hat{a}_{j}-\hat{a}_{j-1}\bigr\Vert_{2}\leq\text{tolerance }$ 
\item Let $K^{n-1}=K_{j}^{n-1}$ 
\end{enumerate}
}
\caption{Implicit Finite Difference Scheme\label{alg:Implicit-Finite-Difference}}
\end{algorithm}

\subsubsection{Logarithmic utility function}

In the 1-asset example, we use the logarithmic utility function and the penalty
function  $\lambda_0 F(\sigma_{t}^{2})=\lambda_{0}(\sigma_{t}^{2})^{2}$. The
terminal condition is given by the utility function, 
\begin{alignat*}{1}
\bar{v}(t_{N},x_{i}) & =U(x_{i})\quad\forall i\in[1,M].
\end{alignat*}

The boundary conditions $\bar{v}(t_{n},x_{1})$ and $\bar{v}(t_{n},x_{M})$
for $n\in[0,N-1]$ are given explicitly by the equation 
\[
\bar{v}(t_{n},x)=\ln(x)+\sup_{\alpha}\inf_{\sigma^{2}} \Bigl\{ \sum_{s=n}^{N-1}\left(\alpha_{s}(\mu-r)+r-\frac{1}{2}\alpha_{s}^{2}\sigma_{s}^{2}+\lambda_{0}(\sigma_{s}^{2})^{2}\right)\Delta t \Bigr\},
\]
with 
\[
\hat{\alpha}_{s}=\frac{\mu-r}{\sigma_{s}^{2}},\:\hat{\sigma}_{s}^{2}=\frac{\alpha^{2}}{4\lambda_{0}}.
\]

\begin{figure}
\centering
\begin{minipage}{.5\textwidth}
  \centering
  \includegraphics[width=.8\linewidth]{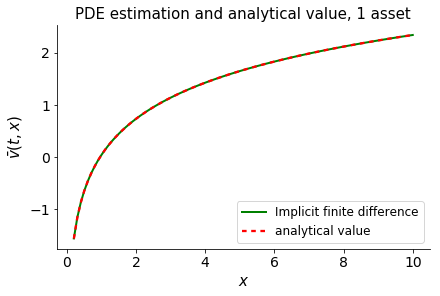} 
  \captionof{figure}{$U(X_{T})=\ln(X_{T})$, $S_{t}\in\mathbb{R}^{1}$}
  \label{fig:PDE-result}
\end{minipage}%
\begin{minipage}{.5\textwidth}
  \centering
  \includegraphics[width=.8\linewidth]{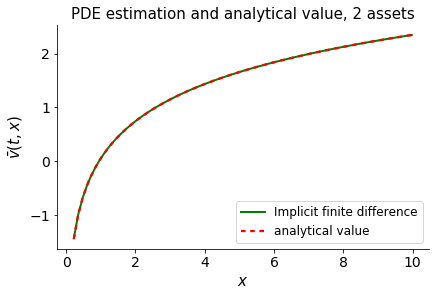}
  \captionof{figure}{$U(X_{T})=\ln(X_{T})$, $S_{t}\in\mathbb{R}^{2}$}
  \label{fig:PDE for 2d log}
\end{minipage}
\end{figure}


Similarly, we can also implement the above method on a 2-asset example
where $S_{t}\in\mathbb{R}^{2}$ and $\lambda_0 F(\Sigma_{t})=\lambda_{0}\left\Vert \Sigma_{t}\right\Vert _{2}^{2}$.  
The HJBI equation becomes
\begin{multline}
\bar{v}_{t}+\inf_{\sigma_{1},\sigma_{2},\rho}\sup_{\alpha_{1},\alpha_{2}} \Bigl\{  \left(\alpha_{1}(\mu_{1}-r)+\alpha_{2}(\mu_{2}-r)+r\right)x\bar{v}_{x}+\frac{1}{2}(\alpha_{1}^{2}\sigma_{1}^{2}+2\alpha_{1}\alpha_{2}\sigma_{1}\sigma_{2}\rho+\alpha_{2}^{2}\sigma_{2}^{2})x^{2}\bar{v}_{xx} \\
+\lambda_{0}(\sigma_{1}^{4}+2\sigma_{1}^{2}\sigma_{2}^{2}\rho^{2}+\sigma_{2}^{4})\Bigr\} =0.
\label{HJBI 2 asset}
\end{multline}

We can solve for the optimal controls  $\hat{\alpha}_{1}, \hat{\alpha}_{2}, \hat{\sigma}_{1},\hat{\sigma}_{2}, \hat{\rho}$ in \eqref{HJBI 2 asset} using the first order condition.
In this example, we always have the optimal $\hat{\sigma}_1, \hat{\sigma}_2 >0$ and $\hat{\rho} \in [-1,1]$. Then, by applying \textbf{Algorithm \ref{alg:Implicit-Finite-Difference}}, we can get the value function of a portfolio with 2 risky assets.

Figure \ref{fig:PDE-result} shows the PDE estimated $\bar{v}(t,x)$ for the 1-asset example with parameters $r=0.015, \mu=0.035, \lambda_{0}=10$; 
Figure \ref{fig:PDE for 2d log} shows result for the 2-asset case with parameters $r=0.015,\mu_{1}=0.035,\mu_{2}=0.045,\lambda_{0}=10$.
Comparing with the analytical solution, we can see that the two curves completely overlap for both 1-asset and 2-asset cases, which validates the accuracy of the PDE approach.

\subsubsection{Power utility function}

In the second example, we use a power utility function. This time,
we only have the terminal condition and the boundary condition for
$x_{1}=0$, but not the boundary condition for a large $x_{M}$. For
functions $x^{\gamma}$ where $\gamma<1,\gamma\neq0$, the limit of
the first order derivative approaches $0$ as $x$ goes to infinity.
Therefore we can use a zero Neumann boundary condition when $x_{M}$
is large. Then we have the following terminal and boundary conditions:
\begin{gather*}
\bar{v}(t_{N},x_{i})=U(x_{i})\,\forall i\in[1,M],\quad\bar{v}(t_{n},x_{1})=0\,\forall x_{1}=0,n\in[0,N-1],\quad\frac{\partial\bar{v}}{\partial x}(t_{n},x_{M})=0\,\forall n\in[0,N-1].
\end{gather*}
Figure \ref{fig:PDE value power utility} shows the simulated value
$\bar{v}(t,x)$ for a range of $x$, with $U(X_{T})=\frac{4}{3}X_{T}^{\frac{1}{4}}$
and parameters $\mu=0.035,r=0.015,\lambda_{0}=10$. We only display
the estimated curve computed by our PDE method, as there is no analytical
solution available for comparison in this example. Figure \ref{fig:estimated a}
shows the first four iterations of the estimated $\hat{a}$ from an
initial guess. There is almost no difference between the four curves,
indicating that the fixed point iteration scheme has converged within
the first four iterations.

%

\begin{figure}%
    \centering
    \subfloat[estimated value function\label{fig:PDE value power utility}]{{\includegraphics[width=.45\textwidth]{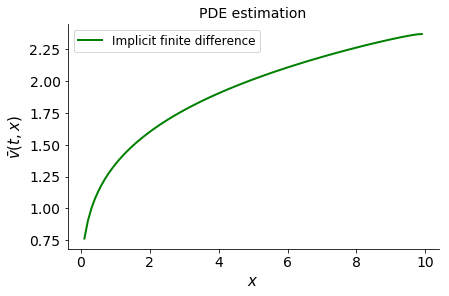} }}%
    \qquad \quad
    \subfloat[estimated $\hat{a}$ in each iteration\label{fig:estimated a}]{{\includegraphics[width=.45\textwidth]{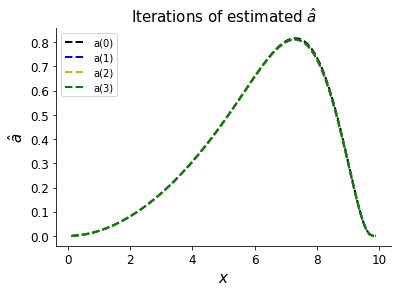} }}%
    \caption{Finite Difference Method for $U(X_{T})=\frac{4}{3}X_{T}^{\frac{1}{4}}$}%
    \label{fig:example}%
\end{figure}

This subsection has shown that the PDE method converges to the true
value efficiently. Nevertheless, there are a few shortcomings to this
approach: 
\begin{itemize}
\item The PDE approach requires tedious algebraic manipulation before implementation.
In particular, even when using the same utility function, the preliminary
computations have to be redone if we switch to a different penalty
function. 
\item In general, PDE approaches suffer from the curse of dimensionality.
As the dimension of the problem becomes higher, the computational
complexity increases exponentially and the approach becomes infeasible.
Although the PDE approach suffices for our current problem as the
wealth process is only one-dimensional, it may not be feasible for
other problems arising from multidimensional stochastic differential
games. 
\end{itemize}
For these two reasons, in the next subsection we develop a numerical
scheme based on Monte Carlo simulations, which can be potentially
useful for high-dimensional problems or in the case of complex penalty
functions.

\subsection{Monte Carlo method \label{subsec:Monte-Carlo-Method} }

In this section, we implement a Regression Monte Carlo scheme to solve
the same robust portfolio allocation problems. \citet{carriere1996valuation}
introduced the Regression Monte Carlo approach to solve optimal stopping
problems for any Markovian process in discrete time. In particular,
he used non-parametric regression techniques. Later, \citet{tsitsiklis2001regression}
and \citet{longstaff2001valuing} used a similar scheme with ordinary
least squares (a.k.a.\! Least Squares Monte Carlo) to value American
options, respectively by value iteration and by performance iteration
(see for example \citealt{denault2017simulation}). Since then, Regression
Monte Carlo has become a popular tool in option pricing and more generally
for solving discrete-time stochastic control problems in finite horizon.

First of all, we discretize the time interval $[0,T]$ into $N$ time
steps with a constant step size $\Delta t=\frac{T}{N}$. Using the
Euler scheme on the logarithm of the state variable, one obtains the
following dynamics for the discrete-time wealth $X_{n}$: 
\begin{alignat}{1}
X_{0} & =x\nonumber \\
X_{n+1} & =X_{n}\exp\left(\left[(\alpha_{n}^{\intercal} (\mu - \mathbf{r}) + r -\frac{1}{2} \alpha_{n}^{\intercal} \Sigma_{n}\alpha_{n} \right]\Delta t+\alpha_{n}^{\intercal} (\Sigma_{n})^{\frac{1}{2}} \Delta W_{n}\right)\,\,,\,n\in[0,N-1]\label{eq:dynamic_discrete}
\end{alignat}

and the discretized form of our value is 
\begin{alignat}{1}
\bar{v}(0,X_{0}) & =\adjustlimits\sup_{\alpha\in\mathcal{A}}\inf_{\Sigma \in\mathcal{B}}\Bigl\{\mathbb{E}\bigl[\lambda_{0}\sum_{n=0}^{N-1}F(\Sigma_{n})\Delta t+U(X_{N})\bigl|X_{0}=x\bigr]\Bigr\}\,.\label{eq:value_discrete}
\end{alignat}
As we have proved in \prettyref{sec:The-Dynamic-Programming}, this
value function satisfies the DPP: 
\begin{alignat}{1}
\bar{v}(N,X_{N}) & =U(X_{N})\nonumber \\
\bar{v}(n,X_{n}) & =\adjustlimits\sup_{\alpha\in\mathcal{A}}\inf_{\Sigma \in\mathcal{B}}\Bigl\{\lambda_{0}F(\Sigma_{n})\Delta t+\mathbb{E}\bigl[\bar{v}(n+1,X_{n+1})\bigl|\mathcal{F}_{n}\bigr]\Bigr\}\,\,,\,n\in[0,N-1]\,.\label{eq:DPP_discrete}
\end{alignat}

\subsubsection{Control randomization}

Inspired by the Dynamic Programming Principle, we can start from the
known terminal condition and compute the value functions backward
in time recursively. Equation \eqref{eq:DPP_discrete} involves a
conditional expectation, which cannot be computed explicitly. Instead,
one can for example use a least squares regression to approximate $\mathbb{E}\bigl[\bar{v}(n+1,X_{n+1})\bigl|\mathcal{F}_{n}\bigr]$
with a polynomial basis function. The obstacle in the implementation
is that we are not able to simulate the paths $X_{n}$ forward, since
the dynamics of the state variable depends on the uncertain controls.
Following \citet{kharroubi2014numerical}, one way to tackle this
problem is an initial randomization of the controls, i.e., we choose
an arbitrary initial distribution for the controls and simulate the
$X_{n}$ with these dummy $\alpha_{n}$ and $\Sigma_{n}$ , before
including these dummy controls in the regressors of the least-squares
regressions.

Proofs of the convergence and error bounds for standard Regression
Monte Carlo are available in \citet{clement2002analysis} and \citet{beutner2013fast}
for example. In the case of controlled dynamics, \citet{kharroubi2015discretization}
analyzed the time-discretization error, and \citet{kharroubi2014numerical}
investigated the projection error generated by approximating the conditional
expectation by basis functions for the control randomization scheme.
Recently, alternative randomization schemes have been proposed in
the literature, such as \citet{ludkovski2019simulation}, \citet{balata2018regress},
\citet{hure2018applications} or \citet{shen2019bsbu}, which are
more amenable to comprehensive convergence proofs, see \citet{balata2017regress}
and \citet{hure2018convergence}. Nevertheless, the classical control
randomization scheme retains some advantages, such as the ease
with which it can handle switching costs, as shown in \citet{zhang2019simulation}.

For the choice of basis function $\phi$, we can use a polynomial
function in $X_{n},\alpha_{n},\Sigma_{n}$, and let $\phi=\sum_{k=0}^{K}\beta_{k}\phi_{k}$.
Once we complete the regression, we can approximate the conditional
expected value function $\mathbb{E}\bigl[\bar{v}(n+1,X_{n+1})\bigl|\mathcal{F}_{n}\bigr]$
in \eqref{eq:DPP_discrete} by $\phi(\hat{\beta};X_{n},\alpha_{n},\Sigma_{n})$.
For the $m$th simulation path, we can find the optimal controls by:
\begin{alignat*}{1}
\hat{\Sigma}_{n}^{m} & =\arg\min_{\Sigma_{n}^{m}}\Bigl\{\lambda_{0}F(\Sigma_{n}^{m})\Delta t+\phi(\hat{\beta};X_{n}^{m},\alpha_{n}^{m},\Sigma_{n}^{m})\Bigr\},\\
\hat{\alpha}_{n}^{m} & =\arg\max_{\alpha_{n}^{m}}\Bigl\{\lambda_{0}F(\hat{\Sigma}_{n}^{m})\Delta t+\phi(\hat{\beta};X_{n}^{m},\alpha_{n}^{m},\hat{\Sigma}_{n}^{m}(\alpha_{n}^{m}))\Bigr\}.
\end{alignat*}
The complete process is shown in \textbf{Algorithm \ref{alg:Control-Randomization}.}

\begin{algorithm}[h]
\textbf{Backward Regression:} 
\begin{enumerate}
\item Choose an initial distribution and generate initial random controls
accordingly.
\item Generate $M$ paths of state variable $X_{n}$. The $m$th path starts
from the initial condition $X_{0}^{m}=x$, evolves following the dynamics
with $\{\alpha_{n}^{m},\Sigma_{n}^{m}\}_{n=0}^{N-1}$ and assign $\bar{v}(N,X_{N}^{m})=U(X_{N}^{m})$. 

\item  For $n=N-1:0$ \textbf{do}
\begin{enumerate}
\item Regress $\bigl\{\bar{v}(n+1,X_{n+1}^{m})\bigr\}_{m=1}^{M}$ on $\bigl\{ X_{n}^{m},\alpha_{n}^{m},\Sigma_{n}^{m}\bigr\}_{m=1}^{M}$
, and get the regression coefficients $\bigl\{\hat{\beta}_{n+1}^{k}\bigr\}_{k}$ 
\item Find the optimal controls $\hat{\alpha}_{n}^{m},\hat{\Sigma}_{n}^{m}$
by $\arg\max_{\alpha}\min_{\Sigma}\Bigl\{\lambda_{0}F(\Sigma_{n}^{m})\Delta t+\sum_{k=0}^{K}\hat{\beta}_{n+1}^{k}\phi_{k}(X_{n,}^{m}\alpha_{n}^{m},\Sigma_{n}^{m})\Bigr\}$ 
\item The value function at time step $n$ is $\bar{v}(n,X_{n}^{m})=\lambda_{0}F(\hat{\Sigma}_{n}^{m})\Delta t+\sum_{k=0}^{K}\hat{\beta}_{n+1}^{k}\phi_{k}(X_{n,}^{m}\hat{\alpha}_{n}^{m},\hat{\Sigma}_{n}^{m})$ 
\end{enumerate}

\item The value function $\bar{v}(0,x)=\frac{1}{M}\sum_{m=1}^{M}\bar{v}(0,X_{0}^{m})$ 
\end{enumerate}
$\,$

\textbf{Forward Resimulation:} 
\begin{enumerate}
\item Set the initial condition $\tilde{X}_{0}^{m}=x$ 
\item For $n=0:N-1$ 
\begin{enumerate}
\item Find the optimal controls $\tilde{\alpha}_{n}^{m},\tilde{\Sigma}_{n}^{m}$
by $\arg\max_{\alpha}\min_{\Sigma}\Bigl\{\lambda_{0}F(\Sigma_{n})\Delta t+\sum_{k=0}^{K}\hat{\beta}_{n+1}^{k}\phi_{k}(\tilde{X}_{n,}^{m}\alpha_{n},\Sigma_{n})\Bigr\}$,
using the regression coefficients obtained in the backward part and
the new state variable $\tilde{X}_{n}^{m}$. 
\item The state variable at time step $n+1$ is 
$\tilde{X}_{n+1}^{m}=\tilde{X}_{n}^{m}\exp\biggl\{\Bigl[(\tilde{\alpha}_{n}^{m})^{\intercal}(\mu- \mathbf{r})+r - \dfrac{1}{2}(\tilde{\alpha}_{n}^{m})^{\intercal} \tilde{\Sigma}_{n}^{m} \tilde{\alpha}_{n}^{m}   \Bigr]\Delta t+(\tilde{\alpha}_{n}^{m})^{\intercal}(\tilde{\Sigma}_{n}^{m})^{\frac{1}{2}}\Delta W_{n}\biggr\}$ 
\end{enumerate}
\item The forward simulated value function $\bar{v}_{f}(0,x)=\frac{1}{M}\sum_{m=1}^{M}\Bigl[\lambda_{0}\sum_{n=0}^{N-1}F(\tilde{\Sigma}_{n}^{m})\Delta t+U(\tilde{X}_{N}^{m})\Bigr]$\caption{Control Randomization\label{alg:Control-Randomization}}
\end{enumerate}
\end{algorithm}

\subsubsection{Logarithmic utility function}

We first consider an example with 1 risky asset. When the utility
function is logarithmic and the penalty function is $\lambda_0 F(\sigma_{t}^{2})=\lambda_{0}(\sigma_{t}^{2})^{2}$, 
we choose the following basis function 
\[
\sum_{k=0}^{K}\beta_{n+1}^{k}\phi_{k}(X_{n,}\alpha_{n},\sigma_{n})=\beta_{0}+\beta_{1}\ln(X_{n})+\beta_{2}\alpha_{n}+\beta_{3}\alpha_{n}\sigma_{n}+\beta_{4}\sigma_{n}^{2}\alpha_{n}^{2}.
\]
To find the optimal controls, we differentiate $\lambda_{0}F(\sigma_{n}^{2})\Delta t+\sum_{k=0}^{K}\beta_{n+1}^{k}\phi_{k}(X_{n,}\alpha_{n},\sigma_{n})$
with respect to $\alpha_{n}$ and $\sigma_{n}^{2}$, then we can get
the optimal controls by solving the following polynomial equation
\[
4\lambda_{0}dt\hat{\sigma}_{n}^{6}+\frac{\beta_{2}\beta_{3}}{2\beta_{4}}\hat{\sigma}_{n}+\frac{\beta_{2}^{2}}{2\beta_{4}}=0.
\]
With $\beta_{4}<0$, there exists a real positive root. We can see
the optimal controls are constants for each step, being independent
of the state variable $X_{n}$, this is the same as our observation
in the analytical solution.

We used $M=5\times10^{6}$ paths, $T=1$ and step size $\Delta t=\frac{1}{50}$
in the simulation, with the parameters $x_0 = 5, r = 0.015, \lambda_0=10$. Figure \ref{fig:analytical function value x05}
shows the backward regression values, forward resimulation values and
true values as we change the parameter $\mu$. Figure \ref{fig:Compare-forward-MC and PDE}
compares the forward resimulation values, finite difference results and
true values as we change the parameter $\mu$.  It shows that both the PDE
and Monte Carlo approach the true value in this example.

\begin{figure}[h]
\begin{minipage}[t]{0.5\columnwidth}%
\includegraphics[scale=0.45]{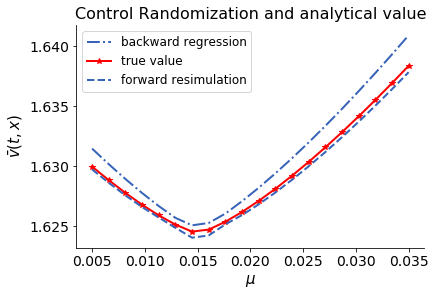}\caption{$U(X_{T})=\ln(X_{T}), S_{t}\in\mathbb{R}^1$ \label{fig:analytical function value x05}}
\end{minipage}%
\begin{minipage}[t]{0.5\columnwidth}%
\includegraphics[scale=0.45]{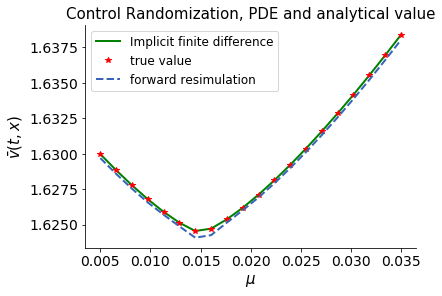}\caption{$U(X_{T})=\ln(X_{T}), S_{t}\in\mathbb{R}^1$ \label{fig:Compare-forward-MC and PDE}}
\end{minipage}
\end{figure}

\begin{figure}[h]
\begin{minipage}[t]{0.5\columnwidth}%
\includegraphics[scale=0.45]{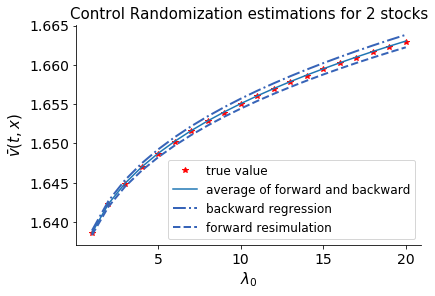}
\caption{$U(X_{T})=\ln(X_{T})$, $S_{t}\in\mathbb{R}^{2}, x_{0}=5,$\\ $r=0.015,\mu_{1}=0.035,\mu_{2}=0.045$ 
\label{fig:MC for 2 stocks log}}
\end{minipage}%
\begin{minipage}[t]{0.5\columnwidth}%
\includegraphics[scale=0.45]{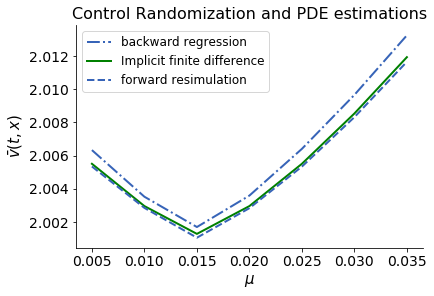} 
\caption{$U(X_{T})=\frac{4}{3}X_{T}^{\frac{1}{4}}$\label{fig:power MC 5million}}
\end{minipage}
\end{figure}

For the example with $2$ risky assets, we use the logarithmic utility function
and the penalty function $F(\Sigma_{t})=\lambda_{0}\left\Vert \Sigma_{t}\right\Vert _{2}^{2}$.
We choose the following basis function in this case:
\begin{multline*}
\sum_{k=0}^{K}\beta_{n+1}^{k}\phi_{k}(X_{n}, \alpha_{n}^1, \alpha_{n}^2, \sigma_{n}^1,  \sigma_{n}^2, \rho_n)
=\beta_{0}+\beta_{1}\ln(X_{n})+\beta_{2}\alpha_n^{1}+\beta_{3}\alpha_n^{2}+\beta_{4}(\alpha_n^{1})^{2} (\sigma_n^{1})^{2} + \beta_{5}(\alpha_n^{2})^{2} (\sigma_n^{2})^{2} \\
+\beta_{6}\alpha_n^{1}\alpha_n^{2}\sigma_n^{1}\sigma_n^{2}\rho_n+\beta_{7}(\sigma_n^{1})^{4}+\beta_{8}(\sigma_n^{2})^{4}+\beta_{9}(\sigma_n^{1})^{2} (\sigma_n^{2})^{2}\rho_n^{2},
\end{multline*}
where $\sigma_n^1, \sigma_n^2$ are the volatilities of the two assets and $\rho_n$ is the correlation between the assets. We can differentiate $\lambda_{0}\left\Vert \Sigma_{t}\right\Vert _{2}^{2}\Delta t+
\sum_{k=0}^{K}\beta_{n+1}^{k}\phi_{k}(X_{n}, \alpha_{n}^1, \alpha_{n}^2, \sigma_{n}^1,  \sigma_{n}^2, \rho_n)$ to get the optimal controls. In practice, we always have
$\hat{\sigma}_n^1, \hat{\sigma}_n^2 > 0$, but we need to truncate $\hat{\rho}_n $ to $[-1, 1]$.
The optimal controls are also constants for each step as in the 1-asset case. 

In the implementation, we use $M=4\times10^{6}$ paths, $T=1$ and step size $\Delta t=\frac{1}{50}$. The result is provided in Figure \ref{fig:MC for 2 stocks log}. This plot compares the backward
regression values, forward resimulation values and the analytical values, and it shows
how the values change w.r.t. the penalty strength $\lambda_{0}$. From our observation, the average of the forward and backward results yields an even better estimate. 

We can observe from Figure \ref{fig:analytical function value x05}
and \ref{fig:MC for 2 stocks log} that, as claimed in \citet{kharroubi2014numerical},
the value function estimated at the end of the backward loop serves
as an upper bound for the true value, while the one obtained from
the forward resimulation serves as a lower bound and has a smaller
error than the upper bound.

\subsubsection{Power utility function}

Here we show a 1-asset example with power utility. When the utility function is $U(X_{T})=\frac{4}{3}X_{T}^{\frac{1}{4}}$
and the penalty function $\lambda_0 F(\sigma_{t}^{2})=\lambda_{0}(\sigma_{t}^{2})^{2}$,
we choose the basis function 
\begin{equation}
\phi=\beta_{0}+\beta_{1}X_{n}^{\frac{1}{4}}+\beta_{2}X_{n}^{\frac{1}{4}}\alpha_{n}+\beta_{3}X_{n}^{\frac{1}{4}}\alpha_{n}\sigma_{n}+\beta_{4}X_{n}^{\frac{1}{4}}\alpha_{n}^{2}\sigma_{n}^{2}.\label{eq:power basis function}
\end{equation}

To find the optimal controls, we differentiate $\lambda_{0}F(\sigma_{n}^{2})\Delta t+\sum_{k=0}^{K}\beta_{n+1}^{k}\phi_{k}(X_{n,}\alpha_{n},\sigma_{n})$
and then get the polynomial equation \eqref{eq:polynomial for power utility}
for each path. We can see the optimal controls $\hat{\alpha}_{n}$
and $\hat{\sigma}_{n}$ depend on $X_{n}$ in this case. 
\begin{equation}
\beta_{2}^{2}X_{n}^{\frac{1}{4}}+\beta_{2}\beta_{3}X_{n}^{\frac{1}{4}}\sigma_{n}+8\beta_{4}\lambda_{0}dt\sigma_{n}^{6}=0\label{eq:polynomial for power utility}
\end{equation}

Figure \eqref{fig:power MC 5million} shows Monte Carlo and finite
difference approximations for a range of drifts $\mu$, with $x_{0}=5,, r=0.015, \lambda_{0}=10$,
$M=5\times10^{6}$, $N=65$. We can see that the PDE estimates lie
within the Monte Carlo bounds and that the forward simulation values
almost overlap the PDE estimations. Although we do not have the analytical
solution for this power utility case, these plots suggest that we
are able to estimate the true values accurately with both Control
Randomization and Finite Difference.

In both the logarithmic and power utility cases, the forward resimulation
always performs better than the backward loop estimates. That is because
the forward resimulation only suffers from one source of error, the
optimal control estimation, while the backward regression suffers
more directly from regression error (see \citealt{kharroubi2014numerical}).
So the forward simulation result is a better estimator of the true
value and is the one we use for comparison with the analytical and
PDE approaches.

From the results above, we can see that for these robust portfolio
allocation problems with one single risky asset, both PDE and Monte
Carlo methods provide accurate estimates, with the PDE estimates being
slightly better overall. Both methods can be considered for solving
robust portfolio allocation problems in practice. Some difficulties
with the Monte Carlo approach are the choice of the basis and the
number of Monte Carlo paths needed for a stable convergence. Still,
the Monte Carlo would be the method of choice for more realistic portfolio
allocation with multiple risky assets (see \citealt{zhang2019simulation}),
as the PDE approach could quickly become computationally intractable
in this situation.

\subsection{Generative Adversarial Networks \label{subsec:GANs} }

In this section, we devise a GAN-based algorithm to solve the two-player zero-sum differential game. 

Generative Adversarial Networks were introduced in \citet{goodfellow2014generative}. A GAN is a combination of two competing (deep) neural networks: a generator and a discriminator. The generator network tries to generate data that looks similar to the training data, and the discriminator network tries to tell the real data from the fake data. The idea behind GANs is very similar to the robust optimization problem studied in our paper: GANs can be interpreted as minimax games between the generator and the discriminator, whereas our problem is a  minimax game between the agent who controls the portfolio allocation and the market who controls the covariance matrix. Inspired by this connection, we propose the following GAN-based algorithm. 

Our GANs are composed of two neural networks; one generates $\alpha$ ($\alpha$-generator),
the other generates $\sigma$ ($\sigma$-generator). The two networks have conflicting goals,
the $\alpha$-generator tries to maximize the expected utility,
while the $\sigma$-generator wants to minimize the expected utility.
They compete against each other during the training. Because we have
two networks with different objectives, it cannot be trained as
a regular neural network. Each training iteration is divided into
two phases: In the first phase, we train the $\alpha$-generator,
with the loss function $L_{1}=-\mathbb{E}\left[U(X_{T})+\lambda_{0}\int_{t}^{T}F(\sigma^2_{s})ds\right]$.
Then the back-propagation only optimizes the weights of the $\alpha$-generator.
In the second phase, given the output $\alpha$ from the $\alpha$-generator,
we train the $\sigma$-generator with a loss function $L_{2}=\mathbb{E}\left[U(X_{T})+\lambda_{0}\int_{t}^{T}F(\sigma^2_{s})ds\right]$.
During this phase, the weights of the $\alpha$-generator are frozen
and the back-propagation only updates the weights of the $\sigma$-generator.
In a zero-sum game, the $\alpha$-generator and $\sigma$-generator
constantly try to outsmart each other. As training advances, the game
may end up at a Nash Equilibrium. 

A demonstration of the simplified  network architecture is illustrated in Figure \ref{fig:demonstration of the adversarial network}. The blue part on the left of Figure \ref{fig:demonstration of the adversarial network} is the $\alpha$-generator. For each time step $n$, we construct a network ($\mathcal{A}_{n}$), with the input $X_n$ and parameter $\sigma_n$, the network generates output $\alpha_n$. With the dynamics of wealth \eqref{eq:dynamic_discrete}, we can continue this process until we get the terminal wealth $X_N$. Once we get the output $\{\alpha_n\}_{n\in [1,N]}$, we can use them as parameters for the $\sigma$-generator (the green part in the figure). In the $\sigma$-generator, similarly, we have one network ($\mathcal{S}_{n}$) for each time step $n$. With the input $X_n$ and parameter $\alpha_n$, we can generate $\sigma_n$. At the end of this phrase, the sequence $\{\sigma_n\}_{n\in [1,N]}$ will be fed into the $\alpha$-generator as parameters as well.  We have  summarized this training process for 1-asset examples in Algorithm  \ref{alg:Deep_Learning}.

In the implementation, we choose the parameters $T=1, r=0.015, \mu=0.035$.  The training data has a sample size $M =200,000$.  We discretize  the investment process into $N=65$ time steps. The deep neural network for each time step contains $4$ hidden layers, using Leaky ReLU as the activation function. For the $\sigma$ generator, to ensure the positivity of the output, we use Leaky Sigmoid as the activation function of the output layer. It is defined as $\text{LeakySigmoid}_\beta(z ) = \frac{1}{1+e^{-x}}\mathbbm{1}(x \leq \beta) +
\left[ \frac{e^{-\beta}}{(1+e^{-\beta})^2} \times (x-\beta)+\frac{1}{1+e^{-\beta}} \right]  \mathbbm{1}(x > \beta)$. Its shape is similar to Sigmoid, but its range is $[0, +\infty]$.
We train the first $100$ epochs with a  learning rate $5\times10^{-4}$, and then we train another $50$ epochs with a decreased learning rate $1\times10^{-4}$.

\begin{figure}[H]
{\centering

\includegraphics[scale=0.28]{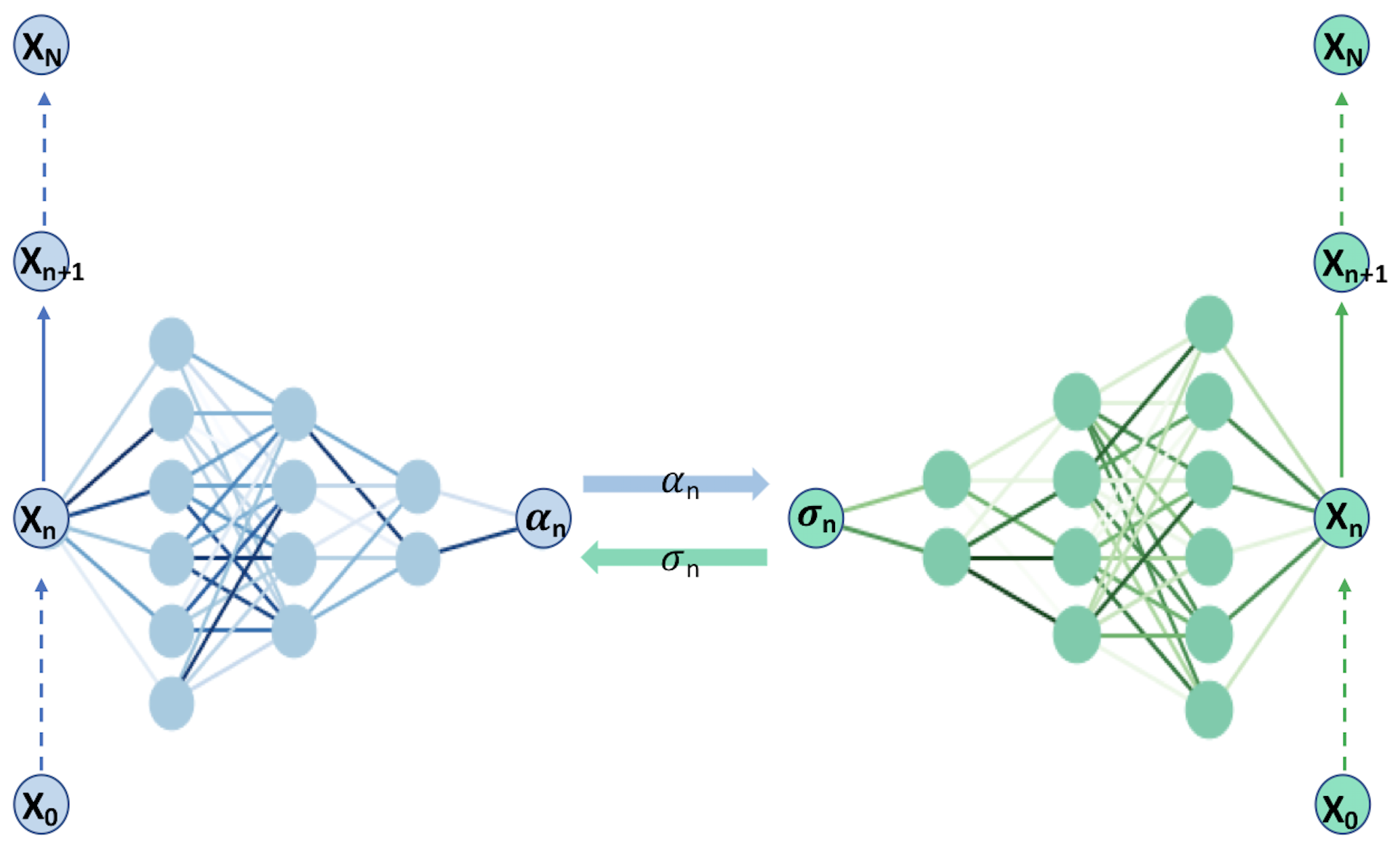}

}
\caption{A demonstration of the adversarial networks\label{fig:demonstration of the adversarial network}}
\end{figure}

\begin{algorithm}[h]
Given the initial condition $X_{0}^{m}=x_{0},\forall m\in[1,M]$,
and the initial starting point of $\sigma=\{\sigma_{n}\}_{n,m=1}^{N-1,M}$:

\For{epoch $=1: \text{number of epochs}$}{
\medskip{}
\textbf{Phase 1}: train the $\alpha$-generator

\For{time step $n=1 : N-1$}{

With the network $\mathcal{A}_{n}$, inputs $\{X_{n}\}_{m=1}^{M}$, parameters $\{\sigma_{n}\}_{m=1}^{M}$,
outputs $\alpha_{n,m}=\mathcal{A}_{n}(X_{n,m},\sigma_{n,m})$;

$X_{n+1,m}=X_{n,m}\exp\left\{ \left(\alpha_{n,m}(\mu-r)+r-\frac{1}{2}\alpha_{n,m}^{2}\sigma_{n,m}^{2}\right)\Delta t+\alpha_{n,m}\sigma_{n,m}\Delta W_{n,m}\right\} $;
}

Loss function $L_{1}=-\frac{1}{M}\sum_{m=1}^{M}\biggl\{ U(X_{N,m})+\lambda_{0}\sum_{n=1}^{N-1}F(\sigma_{n,m}^{2})\Delta t\biggr\}$;

Train the neurons with an Adam optimizer and update $\mathcal{A}_{n}, n\in[1,N-1]$.

\medskip{}

\textbf{Phase 2}: train the $\sigma$-generator

\For{time step $n=1 : N-1$}{

With the network $\mathcal{S}_{n}$, inputs $\{X_{n}\}_{m=1}^{M}$, parameters $\{\alpha_{n}\}_{m=1}^{M}$,
outputs $\sigma_{n.m}=\mathcal{S}_{n}(X_{n,m},\alpha_{n,m})$;

$X_{n+1,m}=X_{n,m}\exp\left\{ \left(\alpha_{n,m}(\mu-r)+r-\frac{1}{2}\alpha_{n,m}^{2}\sigma_{n,m}^{2}\right)\Delta t+\alpha_{n,m}\sigma_{n,m}\Delta W_{n,m}\right\} $;
}

Loss function $L_{2}=\frac{1}{M}\sum_{m=1}^{M}\biggl\{ U(X_{N,m})+\lambda_{0}\sum_{n=1}^{N-1}F(\sigma_{n,m}^{2})\Delta t\biggr\}$;

Train the neurons with an Adam optimizer and update $\mathcal{S}_{n}, n\in[1,N-1]$.

}

\caption{\label{alg:Deep_Learning}Training Generative Adversarial Networks}
\end{algorithm}

We now assess the quality of Algorithm  \ref{alg:Deep_Learning}. Firstly, we use a utility function 
$U(X_{T})=\ln(X_{T})$ and a cost function $\lambda_0 F(\sigma_{t}^{2})=\lambda_0 (\sigma_{t}-\sigma_{0})^{2}$. Assuming the portfolio has an initial wealth $x_{0}=5$, the analytical solution facilitates numerical comparison. 
Figures \ref{fig:GANs_value} compares the learned value functions with the true values for a range of $\lambda_0$. It shows good accuracy of the learned functions versus the true ones. The errors are of magnitude $10^{-5}$. The loss function $L_2$ during the training is presented in Figure \ref{fig:GANs_loss}. Unlike the trend in training regular deep neural networks, the loss function is not monotonically decreasing. As we can see, the minimizer was dominating the competition at the beginning, the loss function decreasing rapidly. Then the maximizer caught up, the loss function increased for a while and finally converged to the true value. 

In the second example, we use a utility function  $U(X_{T})=3X_{T}^{\frac{1}{4}}$ and a cost function $F(\sigma_{t}^{2})=(\sigma_{t}^{2})^2$. We set $\lambda_0 = 10$ in this case and estimate the value functions  for a range of $x_0$.
Since we do not have access to the true values for power utility, we compare the GANs estimated values with the PDE estimations in \ref{fig:GANs_value_power}. The loss function $L_1$ for $x_0 = 6$ during the training is presented in Figure \ref{fig:GANs_loss_power}.

\begin{figure}[H]
\subfloat[value functions estimated with GANs, $U(X_{T})=\ln(X_{T})$\label{fig:GANs_value}]
{\includegraphics[scale=0.45]{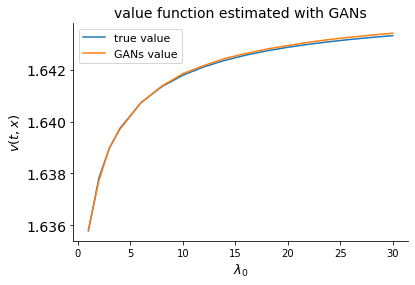}} 
\hfill{}
\subfloat[the minimizer's loss function, $\lambda_{0}=10$ \label{fig:GANs_loss}]
{\includegraphics[scale=0.45]{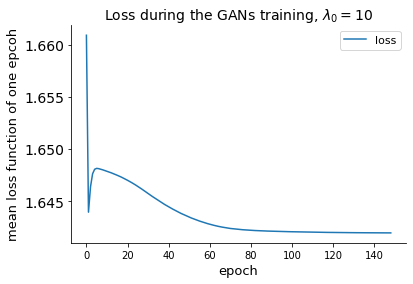}}
\caption{}
\end{figure}

\begin{figure}[H]
\subfloat[value functions estimated with GANs, $U(X_{T})=3X_{T}^{\frac{1}{4}}$ \label{fig:GANs_value_power}]
{\includegraphics[scale=0.45]{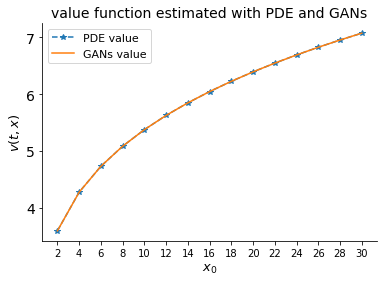}} 
\hfill{}
\subfloat[the maximizer's loss function, $x_{0}=6$ \label{fig:GANs_loss_power}]
{\includegraphics[scale=0.45]{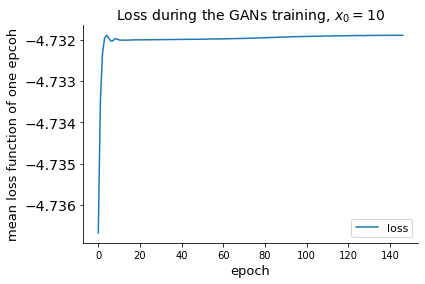}}
\caption{}
\end{figure}

Despite the promising results, a limitation of GANs, shared with deep neural networks in general, is the sensitivity of training to the chosen parameters. On difficult problems, fine-tuning the hyper-parameters of the GAN to facilitate training might require a lot of effort. One standard strategy for stabilizing training is to carefully design the model, 
either by adopting a proper architecture \citep{radford2020unsupervised} or by selecting an easy-to-optimize objective function \citep{salimans2016improved}. In spite of this caveat, GANs can be considered a viable contender to the more classical Monte Carlo methods of subsection \ref{subsec:Monte-Carlo-Method} for robust portfolio allocation involving multiple risky assets, and deserve further investigation.

\section{Conclusion}

In this paper, we interpreted a robust portfolio optimization problem
as a two-player zero-sum stochastic differential game. We have proven that the value function is the unique
viscosity solution of a Hamilton--Jacobi--Bellman--Isaacs equation,
and satisfies the Dynamic Programming Principle. 
We  compared the performance of the robust and non-robust portfolios with both Monte Carlo simulation and empirical market data. Under market shocks, our robust mechanism can prevent huge losses. By choose the $\lambda_0$  properly, the robust portfolios have a higher expected utility than the non-robust one.
In addition to the finite difference method, we provide control randomization and GANs algorithms to estimate the value function. 
These two methods can enrich quantitative techniques for solving robust portfolio optimization problems. 
Both of them have demonstrated high accuracy in the numerical results.   

\section*{Acknowledgements\label{sec:acknowledgements}}

The Centre for Quantitative Finance and Investment Strategies has
been supported by BNP Paribas. Ivan Guo has been partially supported by
the Australian Research Council Discovery Project DP170101227.

\bibliographystyle{chicago}
\bibliography{zerosum_2}

\appendix
\section{Appendices}

\subsection{Proof of Proposition \prettyref{prop:concave in alpha}\label{sec:Appendix-proof in section 4}}
\begin{proof}
First of all, define $w(t,x)\coloneqq\sup_{\alpha\in\mathcal{A}}\mathbb{E}^{t,x}\left[U(X_{T}^{\alpha,{\scriptscriptstyle \Sigma}})\right]$,
$(t,x)\in[0,T]\times\mathbb{R}$. All the assumptions on $\alpha,U,X_{t}$
hold for $w(t,x)$, except that we assume the covariance $\Sigma$
for time $u\in[t,T]$ is a fixed known process in $\mathcal{B}$.
An argument used in \citet[p.52]{pham2009continuous} proved that,
when the utility function $U(\cdot)$ is continuous, increasing and
concave on $\mathbb{R},$ $w(t,\cdot)$ is also increasing and concave
in $x$, $\forall t\in[0,T]$.

For any fixed $\Sigma\in\mathcal{B}$, we define a function $q(t,x)$
by 
\begin{alignat*}{1}
q(t,x) & =\sup_{\alpha\in\mathcal{A}}\mathbb{E}^{t,x}\left[U(X_{T}^{\alpha,{\scriptscriptstyle \Sigma}})+\lambda_{0}\int_{t}^{T}F(\Sigma_{s})ds\right],\\
 & =w(t,x)+\mathbb{E}\bigl[\lambda_{0}\int_{t}^{T}F(\Sigma_{s})ds\bigr].
\end{alignat*}
Then $q(t,x)$ is also concave in $x$ for $t\in[0,T]$. %
We define 
\[
L(t,x,\alpha_{t},\Sigma_{t})\coloneqq\lambda_{0}F(\Sigma_{t})+(\alpha_{t}^{\intercal}\mu+r-\alpha_{t}^{\intercal}\mathbf{r})x\frac{\partial q}{\partial x}(t,x)+\frac{1}{2}tr\left(\alpha_{t}^{\intercal}\Sigma_{t}\alpha_{t}x^{2}\frac{\partial^{2}q}{\partial x^{2}}(t,x)\right).
\]
In addition to Assumption \prettyref{assu:F convex}, we know $L$
is convex in $\Sigma_{t}$ and concave in $\alpha_{t}$. By \citet[Theorem 49.A]{zeidler2013nonlinear},
there exists a saddle point $(\alpha_{t}^{*},\Sigma_{t}^{*})\in A\times B$,
such that 
\begin{equation}
\inf_{\Sigma_{t}\in B}L(t,x,\alpha_{t}^{*},\Sigma_{t})=L(t,x,\alpha_{t}^{*},\Sigma_{t}^{*})=\sup_{\alpha_{t}\in A}L(t,x,\alpha_{t},\Sigma_{t}^{*}).\label{eq:saddle point}
\end{equation}
We know from \citet[Chapter 4.3]{pham2009continuous} that $q(t,x)$
is a viscosity solution of the HJB equation 
\begin{alignat*}{2}
\;\frac{\partial q}{\partial t}(t,x)+\sup_{\alpha_{t}}L(t,x,\alpha_{t},\Sigma_{t})=0, & \quad q(T,x)=U(x) & .
\end{alignat*}

Then $q^{*}(t,x)\coloneqq\sup_{\alpha\in\mathcal{A}}\mathbb{E}^{t,x}\left[U(X_{T}^{\alpha,{\scriptscriptstyle \Sigma^{*}}})+\lambda_{0}\int_{t}^{T}F(\Sigma_{s}^{*})ds\right]$
is a viscosity solution of the PDE 
\[
\frac{\partial q}{\partial t}(t,x)+\sup_{\alpha_{t}\in A}L(t,x,\alpha_{t},\Sigma_{t}^{*})=0,
\]
which is equivalent to 
\begin{equation}
\frac{\partial q}{\partial t}(t,x)+\inf_{\Sigma_{t}\in B}L(t,x,\alpha_{t}^{*},\Sigma_{t})=0\label{eq:inf HJB}
\end{equation}
due to the saddle point property \eqref{eq:saddle point}. Using arguments
similar to the ones in \citet[Chapter 4]{pham2009continuous}, the
function $\inf_{\Sigma\in\mathcal{B}}\mathbb{E}^{t,x}\left[U(X_{T}^{\alpha^{*},{\scriptscriptstyle \Sigma}})+\lambda_{0}\int_{t}^{T}F(\Sigma_{s})ds\right]$
is the unique viscosity solution of the HJB equation \eqref{eq:inf HJB}.
Therefore we have 
\[
\sup_{\alpha\in\mathcal{A}}\mathbb{E}^{t,x}\left[U(X_{T}^{\alpha,{\scriptscriptstyle \Sigma^{*}}})+\lambda_{0}\int_{t}^{T}F(\Sigma_{s}^{*})ds\right]=\inf_{\Sigma\in\mathcal{B}}\mathbb{E}^{t,x}\left[U(X_{T}^{\alpha^{*},{\scriptscriptstyle \Sigma}})+\lambda_{0}\int_{t}^{T}F(\Sigma_{s})ds\right].
\]

With $J(t,x,\alpha,\Sigma)=\mathbb{E}^{t,x}\left[U(X_{T}^{\alpha,{\scriptscriptstyle \Sigma}})+\lambda_{0}\int_{t}^{T}F(\Sigma_{s})ds\right]$,
then the inequality 
\begin{equation}
\inf_{\Sigma\in\mathcal{B}}\sup_{\alpha\in\mathcal{A}}J(t,x,\alpha,\Sigma)\leq\sup_{\alpha\in\mathcal{A}}J(t,x,\alpha,\Sigma^{*})=\inf_{\Sigma\in\mathcal{B}}J(t,x,\alpha^{*},\Sigma)\leq\sup_{\alpha\in\mathcal{A}}\inf_{\Sigma\in\mathcal{B}}J(t,x,\alpha,\Sigma)\label{eq:inequality}
\end{equation}
implies 
\[
\inf_{\Sigma\in\mathcal{B}}\sup_{\alpha\in\mathcal{A}}\mathbb{E}^{t,x}\left[U(X_{T}^{\alpha,{\scriptscriptstyle \Sigma}})+\lambda_{0}\int_{t}^{T}F(\Sigma_{s})ds\right]=\sup_{\alpha\in\mathcal{A}}\inf_{\Sigma\in\mathcal{B}}\mathbb{E}^{t,x}\left[U(X_{T}^{\alpha,{\scriptscriptstyle \Sigma}})+\lambda_{0}\int_{t}^{T}F(\Sigma_{s})ds\right].
\]

From Proposition \ref{prop001}, we have $\underline{u}(t,x)\leq\underline{v}(t,x)\leq\bar{v}(t,x)\leq\bar{u}(t,x)$.
Combining this with $\bar{u}(t,x)=\underline{u}(t,x)$, we obtained
the required equalities 
\[
\underline{u}(t,x)=\underline{v}(t,x)=\bar{v}(t,x)=\bar{u}(t,x).
\]
\end{proof}

\subsection{Proof of Proposition \prettyref{prop:continuous in x}\label{sec:Appendix-proof in section 5}}

\begin{proof}
Let $X_{T}^{{\scriptscriptstyle \Sigma,\Gamma}}$ and $\bar{X}_{T}^{{\scriptscriptstyle \Sigma,\Gamma}}$
be the solutions of the SDE \eqref{eq:dynamic of X} with initial
states $(t,x)$ and $(t,\bar{x})$ respectively, they are both controlled
by an arbitrary pair of admissible control and strategy processes
$(\Sigma,\Gamma)$. From Assumption \prettyref{assu:U}, we have 
\[
\Bigl|U(X_{T}^{{\scriptscriptstyle \Sigma,\Gamma}})-U(\bar{X}_{T}^{{\scriptscriptstyle \Sigma,\Gamma}})\Bigr|\leq Q(\left|X_{T}^{{\scriptscriptstyle \Sigma,\Gamma}}\right|,\left|\bar{X}_{T}^{{\scriptscriptstyle \Sigma,\Gamma}}\right|)\left|X_{T}^{{\scriptscriptstyle \Sigma,\Gamma}}-\bar{X}_{T}^{{\scriptscriptstyle \Sigma,\Gamma}}\right|.
\]
We have 
\begin{alignat}{1}
\Bigl|\mathbb{E}^{t,x}\left[U(X_{T}^{{\scriptscriptstyle \Sigma,\Gamma}})\right]-\mathbb{E}^{t,\bar{x}}\left[U(\bar{X}_{T}^{{\scriptscriptstyle \Sigma,\Gamma}})\right]\Bigr|\leq\mathbb{E}\Bigl|U(X_{T}^{{\scriptscriptstyle \Sigma,\Gamma}})-U(\bar{X}_{T}^{{\scriptscriptstyle \Sigma,\Gamma}})\Bigr|.
\end{alignat}
By the Cauchy-Schwarz inequality, 
\begin{alignat}{1}
\Bigl(\mathbb{E}\Bigl|U(X_{T}^{{\scriptscriptstyle \Sigma,\Gamma}})-U(\bar{X}_{T}^{{\scriptscriptstyle \Sigma,\Gamma}})\Bigr|\Bigr)^{2} & \leq\mathbb{E}\biggl[Q(\bigl|X_{T}^{{\scriptscriptstyle \Sigma,\Gamma}}\bigr|,\bigl|\bar{X}_{T}^{{\scriptscriptstyle \Sigma,\Gamma}}\bigr|)^{2}\biggr]\mathbb{E}\biggl[\left|X_{T}^{{\scriptscriptstyle \Sigma,\Gamma}}-\bar{X}_{T}^{{\scriptscriptstyle \Sigma,\Gamma}}\right|^{2}\biggr].\label{eq:value_2_1}
\end{alignat}
It is straightforward to check that there exist constants $C$, $m_{1},m_{2}$
and $\beta_{0}$ such that 
\begin{equation}
\mathbb{E}\biggl[Q(\bigl|X_{T}^{{\scriptscriptstyle \Sigma,\Gamma}}\bigr|,\bigl|\bar{X}_{T}^{{\scriptscriptstyle \Sigma,\Gamma}}\bigr|)^{2}\biggr]\mathbb{E}\biggl[\left|X_{T}^{{\scriptscriptstyle \Sigma,\Gamma}}-\bar{X}_{T}^{{\scriptscriptstyle \Sigma,\Gamma}}\right|^{2}\biggr]\leq C\mathbb{E}\biggl[\bigl|X_{T}^{{\scriptscriptstyle \Sigma,\Gamma}}\bigr|^{2m_{1}}+\bigl|\bar{X}_{T}^{{\scriptscriptstyle \Sigma,\Gamma}}\bigr|^{2m_{2}}\biggr]e^{2\beta_{0}(T-t)}\left|x-\bar{x}\right|^{2}.\label{eq:estimation_3}
\end{equation}
By the classical inequality $\mathbb{E}^{t,x}\Bigl[\max_{t\leq s\leq T}\bigl|X_{s}^{{\scriptscriptstyle \Sigma,\Gamma}}\bigr|^{2m}\Bigr]\leq{\color{blue}{\color{black}C_{T}}}(1+x^{2m})$
(e.g., \citet[Theorem 1.3.15]{pham2009continuous}), for arbitrary
control and strategy processes $\Gamma,\Sigma$, we have 
\begin{alignat}{1}
\Bigl|\mathbb{E}^{t,x}\left[U(X_{T}^{{\scriptscriptstyle \Sigma,\Gamma}})\right]-\mathbb{E}^{t,\bar{x}}\left[U(\bar{X}_{T}^{{\scriptscriptstyle \Sigma,\Gamma}})\right]\Bigr| & \leq\Phi(\left|x\right|,\left|\bar{x}\right|)\left|x-\bar{x}\right|,\label{eq:v_continuous}
\end{alignat}
where $C_{T},m$ are constants, and $\Phi$ is a polynomial function.

Next, for all bounded functions $\mathbb{E}^{t,x}\Bigl[\lambda_{0}\int_{t}^{T}F(\Sigma_{s})ds+U(X_{T}^{{\scriptscriptstyle \Gamma,\Sigma}})\Bigr]$
and $\mathbb{E}^{t,\bar{x}}\Bigl[\lambda_{0}\int_{t}^{T}F(\Sigma_{s})ds+U(\bar{X}_{T}^{{\scriptscriptstyle \Gamma,\Sigma}})\Bigr]$,
\begin{alignat}{1}
 & \biggl|\inf_{\Sigma}\mathbb{E}^{t,x}\Bigl[\lambda_{0}\int_{t}^{T}F(\Sigma_{s})ds+U(X_{T}^{{\scriptscriptstyle \Gamma,\Sigma}})\Bigr]-\inf_{\Sigma}\mathbb{E}^{t,\bar{x}}\Bigl[\lambda_{0}\int_{t}^{T}F(\Sigma_{s})ds+U(\bar{X}_{T}^{{\scriptscriptstyle \Gamma,\Sigma}})\Bigr]\biggr|\nonumber \\
\leq{} & \sup_{\Sigma}\biggl|\mathbb{E}^{t,x}\Bigl[\lambda_{0}\int_{t}^{T}F(\Sigma_{s})ds+U(X_{T}^{{\scriptscriptstyle \Gamma,\Sigma}})\Bigr]-\mathbb{E}^{t,\bar{x}}\Bigl[\lambda_{0}\int_{t}^{T}F(\Sigma_{s})ds+U(\bar{X}_{T}^{{\scriptscriptstyle \Gamma,\Sigma}})\Bigr]\biggr|,\label{eq:inequality_1}\\
 & \biggl|\sup_{\Gamma}\mathbb{E}^{t,x}\Bigl[\lambda_{0}\int_{t}^{T}F(\Sigma_{s})ds+U(X_{T}^{{\scriptscriptstyle \Gamma,\Sigma}})\Bigr]-\sup_{\Gamma}\mathbb{E}^{t,\bar{x}}\Bigl[\lambda_{0}\int_{t}^{T}F(\Sigma_{s})ds+U(\bar{X}_{T}^{{\scriptscriptstyle \Gamma,\Sigma}})\Bigr]\biggr|\nonumber \\
\leq{} & \sup_{\Gamma}\biggl|\mathbb{E}^{t,x}\Bigl[\lambda_{0}\int_{t}^{T}F(\Sigma_{s})ds+U(X_{T}^{{\scriptscriptstyle \Gamma,\Sigma}})\Bigr]-\mathbb{E}^{t,\bar{x}}\Bigl[\lambda_{0}\int_{t}^{T}F(\Sigma_{s})ds+U(\bar{X}_{T}^{{\scriptscriptstyle \Gamma,\Sigma}})\Bigr]\biggr|.\label{eq:inequality_2}
\end{alignat}
Under Assumptions \prettyref{assu:U} and \prettyref{assu:finiteness},
$\bar{v}(t,x)$ is bounded. Then we can write the difference between
the two value functions as: 
\begin{alignat}{1}
 & \Bigl|\bar{v}(t,x)-\bar{v}(t,\bar{x})\Bigr|\\
\leq{} & \sup_{\Gamma}\sup_{\Sigma}\biggl|\mathbb{E}^{t,x}\Bigl[\lambda_{0}\int_{t}^{T}F(\Sigma_{s})ds+U(X_{T}^{{\scriptscriptstyle \Gamma,\Sigma}})\Bigr]-\mathbb{E}^{t,\bar{x}}\Bigl[\lambda_{0}\int_{t}^{T}F(\Sigma_{s})ds+U(\bar{X}_{T}^{{\scriptscriptstyle \Gamma,\Sigma}})\Bigr]\biggr|.
\end{alignat}

In addition to the inequality \eqref{eq:v_continuous}, the value
function $\bar{v}(t,x)$ is locally Lipschitz continuous in $x$. 
\end{proof}

\subsection{Proof of \prettyref {thm:The-Dynamic-Programming}\label{sec:Appendix-proof of DPP}}

\begin{proof}
We use localization techniques here. Let $B_{k}=\{x\in\mathbb{R},x^{2}<k^{2}\}$,
let $\phi_{k}(x)$ be a function such that $\phi_{k}(x)=1$ on $B_{k}$,
and $\phi_{k}(x)=0$ outside $B_{k}$. Then we can define a new process
\begin{equation}
dX_{s}^{k}=\phi_{k}(X_{s}^{k})X_{s}^{k}\Bigl[(\alpha_{s}^{\intercal}\mu+r-\alpha_{s}^{\intercal}\mathbf{r})ds+\alpha_{s}^{\intercal}\sigma_{s}dW_{s}\Bigr],\label{eq:localSDE-1}
\end{equation}
starting from an initial condition $(t,x)\in[0,T]\times\mathbb{R}$.
Let $U^{k}(x)=\phi_{k+2}(x)U(x)$, then we can define the truncated
value function by 
\begin{equation}
\bar{v}^{k}(t,x)=\adjustlimits\sup_{\Gamma\in\mathcal{N}}\inf_{\Sigma\in\mathcal{B}}\Bigl\{\mathbb{E}^{t,x}\Bigl[\lambda_{0}\int_{t}^{T}F(\Sigma_{s})ds+U^{k}(X_{T}^{k,{\scriptscriptstyle \Gamma,\Sigma}})\Bigr]\Bigr\}.\label{eq:value2}
\end{equation}
In the above setting, the drift and volatility functions in the SDE
\eqref{eq:localSDE-1} are bounded, and the utility function in \eqref{eq:value2}
is bounded and Lipschitz continuous. Since all assumptions of \citet{fleming1989existence}
are satisfied, the localized value function $\bar{v}^{k}$ defined
in \eqref{eq:value2} satisfies the dynamic programming principle:
for $t\leq t+\theta\leq T$, 
\begin{equation}
\bar{v}^{k}(t,x)=\adjustlimits\sup_{\Gamma\in\mathcal{N}}\inf_{\Sigma\in\mathcal{B}}\Bigl\{\mathbb{E}^{t,x}\Bigl[\lambda_{0}\int_{t}^{t+\theta}F(\Sigma_{s})ds+\bar{v}^{k}(t+\theta,X_{t+\theta}^{k,{\scriptscriptstyle \Gamma,\Sigma}})\Bigr]\Bigr\}.\label{eq:DPP2}
\end{equation}
In this proof, $X_{t+\theta}^{k,{\scriptscriptstyle \Gamma,\Sigma}}$
and $X_{t+\theta}^{{\scriptscriptstyle \Gamma,\Sigma}}$ are the solutions
of SDE \eqref{eq:localSDE-1} and SDE \eqref{eq:dynamic of X} respectively,
both starting from $(t,x)$, controlled by processes $\Gamma,\Sigma$
for the time $u\in[t,t+\theta]$.

As $k\rightarrow\infty$, $\bar{v}^{k}(t,x)$ defined in \eqref{eq:value2}
approaches $\bar{v}(t,x)$ defined in \eqref{eq:upper}, then our
problem reduces to proving that the right hand side of \eqref{eq:DPP2}
converges to the right hand side of \eqref{eq:DPP1}.

Note that if $X_{s}^{k}$ is in $\overline{B_{k+1}}$, then $X_{u}^{k}$
is in $\overline{B_{k+1}}$ $\forall u\in[s,T]$ almost surely. \textcolor{black}{Define}
$\tau_{k}$ to be the first exit time of $X_{t}^{k}$ from $B_{k}$.
Thus, for $(t,x)\in[0,T]\times\mathbb{R}$, we have 
\begin{alignat}{1}
 & \biggl|\adjustlimits\sup_{\Gamma\in\mathcal{N}}\inf_{\Sigma\in\mathcal{B}}\Bigl\{\mathbb{E}^{t,x}\Bigl[\lambda_{0}\int_{t}^{t+\theta}F(\Sigma_{s})ds+\bar{v}^{k}(t+\theta,X_{t+\theta}^{k,{\scriptscriptstyle \Gamma,\Sigma}})\Bigr]\Bigr\}-\adjustlimits\sup_{\Gamma\in\mathcal{N}}\inf_{\Sigma\in\mathcal{B}}\Bigl\{\mathbb{E}^{t,x}\Bigl[\lambda_{0}\int_{t}^{t+\theta}F(\Sigma_{s})ds+\bar{v}(t+\theta,X_{t+\theta}^{{\scriptscriptstyle \Gamma,\Sigma}})\Bigr]\Bigr\}\biggr|\nonumber \\
\leq{} & \adjustlimits\sup_{\Gamma\in\mathcal{N}}\sup_{\Sigma\in\mathcal{B}}\mathbb{E}^{t,x}\biggl|\bar{v}^{k}(t+\theta,X_{t+\theta}^{k,{\scriptscriptstyle \Gamma,\Sigma}})-\bar{v}(t+\theta,X_{t+\theta}^{{\scriptscriptstyle \Gamma,\Sigma}})\biggr|\nonumber \\
\leq{} & \adjustlimits\sup_{\Gamma\in\mathcal{N}}\sup_{\Sigma\in\mathcal{B}}\mathbb{E}^{t,x}\biggl|\Bigl(\bar{v}^{k}(t+\theta,X_{t+\theta}^{k,{\scriptscriptstyle \Gamma,\Sigma}})-\bar{v}(t+\theta,X_{t+\theta}^{{\scriptscriptstyle \Gamma,\Sigma}})\Bigr)\mathbbm{1}(\tau_{k}>T)\biggr|\label{eq:convergence-1-1}\\
 & \quad+\adjustlimits\sup_{\Gamma\in\mathcal{N}}\sup_{\Sigma\in\mathcal{B}}\mathbb{E}^{t,x}\biggl|\Bigl(\bar{v}^{k}(t+\theta,X_{t+\theta}^{k,{\scriptscriptstyle \Gamma,\Sigma}})-\bar{v}(t+\theta,X_{t+\theta}^{{\scriptscriptstyle \Gamma,\Sigma}})\Bigr)\mathbbm{1}(\tau_{k}\leq T)\biggr|.\label{eq:convergence-2-1}
\end{alignat}

If $\tau_{k}>T$, the term \eqref{eq:convergence-1-1} is zero. For
the term \eqref{eq:convergence-2-1}, for any arbitrary pair $(\bar{\Gamma},\bar{\Sigma})$,
we have 
\begin{alignat}{1}
 & \left(\mathbb{E}^{t,x}\biggl|\bar{v}^{k}(t+\theta,X_{t+\theta}^{k,{\scriptscriptstyle \bar{\Gamma},\bar{\Sigma}}})-\bar{v}(t+\theta,X_{t+\theta}^{{\scriptscriptstyle \bar{\Gamma},\bar{\Sigma}}})\mathbbm{1}(\tau_{k}\leq T)\biggr|\right)^{2}\nonumber \\
\leq{} & \mathbb{E}^{t,x}\left[\biggl|\bar{v}^{k}(t+\theta,X_{t+\theta}^{k,{\scriptscriptstyle \bar{\Gamma},\bar{\Sigma}}})-\bar{v}(t+\theta,X_{t+\theta}^{{\scriptscriptstyle \bar{\Gamma},\bar{\Sigma}}})\biggr|^{2}\right]\times\mathbb{P}(\tau_{k}\leq T).\label{eq:upper bound in DPP}
\end{alignat}
Finally our task is to show that the upper bound \eqref{eq:upper bound in DPP}
converges to zero as $k$ goes to infinity.

Let $X_{T}^{k,{\scriptscriptstyle \Gamma,\Sigma}}$ be the solution
of SDE \eqref{eq:localSDE-1} starting from $(t+\theta,X_{t+\theta}^{k,{\scriptscriptstyle \bar{\Gamma},\bar{\Sigma}}})$,
and $X_{T}^{{\scriptscriptstyle \Gamma,\Sigma}}$ be the solution
of \eqref{eq:dynamic of X} starting from $(t+\theta,X_{t+\theta}^{{\scriptscriptstyle \bar{\Gamma},\bar{\Sigma}}})$,
they are controlled by $\Gamma,\Sigma$ for the time $u\in[t+\theta,T]$.

Using arguments in equations \eqref{eq:inequality_1} and \eqref{eq:inequality_2},
\begin{alignat*}{1}
 & \Bigl|\bar{v}^{k}(t+\theta,X_{t+\theta}^{k,{\scriptscriptstyle \bar{\Gamma},\bar{\Sigma}}})-\bar{v}(t+\theta,X_{t+\theta}^{{\scriptscriptstyle \bar{\Gamma},\bar{\Sigma}}})\Bigr|\\
\leq{} & \sup_{\Gamma\in\mathcal{N}}\sup_{\Sigma\in\mathcal{B}}\Bigl|\mathbb{E}^{t+\theta,X_{t+\theta}^{k,{\scriptscriptstyle \bar{\Gamma},\bar{\Sigma}}}}\Bigl[\lambda_{0}\int_{t}^{T}F(\Sigma_{s})ds+U^{k}(X_{T}^{k,{\scriptscriptstyle \Gamma,\Sigma}})\Bigr]-\mathbb{E}^{t+\theta,X_{t+\theta}^{{\scriptscriptstyle \bar{\Gamma},\bar{\Sigma}}}}\Bigl[\lambda_{0}\int_{t}^{T}F(\Sigma_{s})ds+U(X_{T}^{{\scriptscriptstyle \Gamma,\Sigma}})\Bigr]\Bigr|
\end{alignat*}
For any arbitrary controls $(\Gamma,\Sigma)$ for the time $u\in[t+\theta,T]$,
it is easy to see that 
\begin{alignat*}{1}
 & \mathbb{E}^{t,x}\Big(\Bigl|\mathbb{E}^{t+\theta,X_{t+\theta}^{k,{\scriptscriptstyle \bar{\Gamma},\bar{\Sigma}}}}\Bigl[U^{k}(X_{T}^{k,{\scriptscriptstyle \Gamma,\Sigma}})\Bigr]-\mathbb{E}^{t+\theta,X_{t+\theta}^{{\scriptscriptstyle \bar{\Gamma},\bar{\Sigma}}}}\Bigl[U(X_{T}^{{\scriptscriptstyle \Gamma,\Sigma}})\Bigr]\Bigr|\Big)\\
\leq{} & \mathbb{E}^{t,x}\Big(\mathbb{E}^{t+\theta,X_{t+\theta}^{k,{\scriptscriptstyle \bar{\Gamma},\bar{\Sigma}}}}\Bigl|U^{k}(X_{T}^{k,{\scriptscriptstyle \Gamma,\Sigma}})\Bigr|+\mathbb{E}^{t+\theta,X_{t+\theta}^{{\scriptscriptstyle \bar{\Gamma},\bar{\Sigma}}}}\Bigl|U(X_{T}^{{\scriptscriptstyle \Gamma,\Sigma}})\Bigr|\Big)\\
\leq{} & \mathbb{E}^{t,x}\Big(\mathbb{E}^{t+\theta,X_{t+\theta}^{k,{\scriptscriptstyle \bar{\Gamma},\bar{\Sigma}}}}\Bigl[C\Bigl|X_{T}^{k,{\scriptscriptstyle \Gamma,\Sigma}}\Bigr|^{2m_{1}}\Bigr]+\mathbb{E}^{t+\theta,X_{t+\theta}^{{\scriptscriptstyle \bar{\Gamma},\bar{\Sigma}}}}\Bigl[C\Bigl|X_{T}^{{\scriptscriptstyle \Gamma,\Sigma}}\Bigr|^{2m_{2}}\Bigr]\Big)\\
\leq{} & \mathbb{E}^{t,x}\Big(K_{T}\bigl(1+\Bigl|X_{t+\theta}^{k,{\scriptscriptstyle \bar{\Gamma},\bar{\Sigma}}}\Bigr|^{2m_{1}}\bigr)+K_{T}\bigl(1+\Bigl|X_{t+\theta}^{{\scriptscriptstyle \bar{\Gamma},\bar{\Sigma}}}\Bigr|^{2m_{2}}\bigr)\Big)\\
\leq{} & C_{T}\Big(1+|x|^{2m}\Bigr),
\end{alignat*}
where $C,K_{T},C_{T},m_{1},m_{2},m$ are constants. Then there exists
a polynomial $\Phi$ such that 
\begin{alignat}{1}
\mathbb{E}^{t,x}\left[\Bigl|\bar{v}^{k}(t+\theta,X_{t+\theta}^{k,{\scriptscriptstyle \Gamma,\Sigma}})-\bar{v}(t+\theta,X_{t+\theta}^{{\scriptscriptstyle \Gamma,\Sigma}})\Bigr|^{2}\right] & \leq\Phi(|x|),\label{poly estimation}
\end{alignat}

and the Markov inequality yields 
\begin{equation}
\mathbb{P}(\tau_{k}\leq T)\leq\frac{\mathbb{E}^{t,x}\Bigl[\sup_{t\leq s\leq T}\bigl|X_{s}^{{\scriptscriptstyle \Gamma,\Sigma}}\bigr|^{2}\Bigr]}{k^{2}}\leq\frac{C_{T}\left(1+x^{2}\right)}{k^{2}},\label{Markov}
\end{equation}
where $C_{T}$ is a constant independent of $k$. Therefore we have
\[
\mathbb{E}^{t,x}\biggl|(\bar{v}^{k}(t+\theta,X_{t+\theta}^{k,{\scriptscriptstyle \Gamma,\Sigma}})-\bar{v}(t+\theta,X_{t+\theta}^{{\scriptscriptstyle \Gamma,\Sigma}}))\mathbbm{1}(\tau_{k}\leq T)\biggr|\leq\frac{K(\bigl|x\bigr|)}{k},
\]
where $K(\bigl|x\bigr|)$ is a polynomial function in terms of $x$.

As $k\rightarrow\infty$, the term \eqref{eq:convergence-2-1} goes
to zero as well, therefore 
\[
\bar{v}(t,x)=\adjustlimits\sup_{\Gamma\in\mathcal{N}}\inf_{\Sigma\in\mathcal{B}}\Bigl\{\mathbb{E}^{t,x}\Bigl[\lambda_{0}\int_{t}^{t+\theta}F(\Sigma_{s})ds+\bar{v}(t+\theta,X_{t+\theta}^{{\scriptscriptstyle \Gamma,\Sigma}})\Bigr]\Bigr\},
\]
as the left and right hand sides of \eqref{eq:DPP2} converge to the
left and right hand sides of equation \eqref{eq:DPP1} respectively. 
\end{proof}

\subsection{Proof of Corollary \ref{cor:continuous in time}\label{sec:Appendix-proof of continuous in time }}
\begin{proof}
Let $X_{s}^{{\scriptscriptstyle \Gamma,\Sigma}}$ be the solution
of the SDE \eqref{eq:dynamic of X} starting from $x$ at time $t$,
controlled by $\Gamma,\Sigma$ for time $u\in[t,s]$. By the Dynamic
Programming Principle and inequality \eqref{eq:v_continuous}, for
$t<s<T,$ 
\begin{alignat*}{1}
\bigl|\bar{v}(t,x)-\bar{v}(s,x)\bigr| & =\Bigl|\adjustlimits\sup_{\Gamma\in\mathcal{N}}\inf_{\Sigma\in\mathcal{B}}\Bigl\{\mathbb{E}^{t,x}\left[\lambda_{0}\int_{t}^{s}F(\Sigma_{u})du+\bar{v}(s,X_{s}^{{\scriptscriptstyle \Gamma,\Sigma}})\right]\Bigr\}-\bar{v}(s,x)\Bigr|.
\end{alignat*}
With any arbitrary control and strategy processes $(\hat{\Sigma},\hat{\Gamma})$
for time $u\in[t,s]$, we have 
\begin{alignat}{1}
 & \Bigl|\mathbb{E}^{t,x}\left[\lambda_{0}\int_{t}^{s}F(\hat{\Sigma}_{u})du+\bar{v}(s,X_{s}^{{\scriptscriptstyle \hat{\Gamma},\hat{\Sigma}}})\right]-\bar{v}(s,x)\Bigr|\\
={} & \Bigl|\mathbb{E}^{t,x}\left[\lambda_{0}\int_{t}^{s}F(\hat{\Sigma}_{u})du\right]+\mathbb{E}^{t,x}\left[\bar{v}(s,X_{s}^{{\scriptscriptstyle \hat{\Gamma},\hat{\Sigma}}})-\bar{v}(s,x)\right]\Bigr|\label{eq:continuous in t}\\
\leq{} & \Bigl|\mathbb{E}^{t,x}\left[\lambda_{0}\int_{t}^{s}F(\hat{\Sigma}_{u})du\right]\Bigr|+\mathbb{E}^{t,x}\left[\sup_{\Gamma}\sup_{\Sigma}\left(\mathbb{E}^{s,X_{s}^{{\scriptscriptstyle \hat{\Gamma},\hat{\Sigma}}}}[U(X_{T}^{{\scriptscriptstyle \Gamma,\Sigma}})]-\mathbb{E}^{s,x}[U(X_{T}^{{\scriptscriptstyle \Gamma,\Sigma}}]\right)\right].\label{eq:continuous in t-2}
\end{alignat}
Referring to \eqref{eq:v_continuous}, there exist a polynomial function
$\Phi$ and constants $C,C_{T},m_{1},m_{2}$ for an arbitrary pair
of $(\Gamma,\Sigma)$ for time $u\in[s,T]$ such that 
\begin{alignat*}{1}
 & \mathbb{E}^{t,x}\left[\mathbb{E}^{s,X_{s}^{{\scriptscriptstyle \hat{\Gamma},\hat{\Sigma}}}}[U(X_{T}^{{\scriptscriptstyle \Gamma,\Sigma}})]-\mathbb{E}^{s,x}[U(X_{T}^{{\scriptscriptstyle \Gamma,\Sigma}})]\right]\\
\leq{} & \mathbb{E}^{t,x}\Bigl[\Phi(\left|X_{s}^{{\scriptscriptstyle \hat{\Gamma},\hat{\Sigma}}}\right|,\bigl|x\bigr|)\left|X_{s}^{{\scriptscriptstyle \hat{\Gamma},\hat{\Sigma}}}-x\right|\Bigr]\\
\leq{} & \mathbb{E}^{t,x}\left[\Phi(\left|X_{s}^{{\scriptscriptstyle \hat{\Gamma},\hat{\Sigma}}}\right|,\bigl|x\bigr|)^{2}\right]^{1/2}\mathbb{E}^{t,x}\left[\left|X_{s}^{{\scriptscriptstyle \hat{\Gamma},\hat{\Sigma}}}-x\right|^{2}\right]^{1/2}\\
\leq{} & \Bigl(Cx^{2m_{1}}+C_{T}\left(1+x^{2m_{2}}\right)\Bigr)^{1/2}\mathbb{E}^{t,x}\Bigl[\bigl|X_{s}^{{\scriptscriptstyle \hat{\Gamma},\hat{\Sigma}}}-x\bigr|^{2}\Bigr]^{1/2}.
\end{alignat*}
We know 
\[
\mathbb{E}^{t,x}\Bigl[\bigl|X_{s}^{{\scriptscriptstyle \hat{\Gamma},\hat{\Sigma}}}-x\bigr|^{2}\Bigr]\leq C_{T}(1+x^{2})(s-t).
\]
Let $\eta=\max\left\{ \bigl|F(\Sigma_{u})\bigr|:\Sigma_{u}\in B\right\} $,
therefore 
\[
\bigl|\bar{v}(t,x)-\bar{v}(s,x)\bigr|\leq\lambda_{0}\eta(s-t)+\Phi(|x|)(s-t)^{1/2}.
\]
Hence $\bar{v}(t,x)$ is H\"older continuous in $t\in[0,T]$. 
\end{proof}

\subsection{Proof of \prettyref{thm:existence-of-viscosity}\label{sec:Appendix-proof of existence of viscosity solution}}

\begin{proof}
We again make use of the localized processes $X_{t}^{k},U^{k}$ and
$\bar{v}^{k}$ from the proof of \prettyref{thm:The-Dynamic-Programming}
in \prettyref{sec:The-Dynamic-Programming}. The HJBI equation associated
with SDE \eqref{eq:localSDE-1} is 
\begin{align}
\begin{cases}
\frac{\partial v}{\partial t}(t,x)+H^{k}(t,x,\frac{\partial v}{\partial x}(t,x),\frac{\partial^{2}v}{\partial x^{2}}(t,x))=0 & \text{in }\,[0,T)\times\mathbb{R},\\
v(T,x)=U^{k}(x) & \text{on }\,[T]\times\mathbb{R},
\end{cases}\label{eq:HJBI2}
\end{align}
where 
\begin{equation}
H^{k}(t,x,p,M)=\adjustlimits\inf_{\mathbf{\Sigma}\in B}\sup_{\mathbf{a}\in A}\left\{ \lambda_{0}F(\mathbf{\Sigma})+\phi_{k}(x)(\mathbf{a}^{\intercal}\mu+r-\mathbf{a}^{\intercal}\mathbf{r})xp+\frac{1}{2}tr\left(\phi_{k}^{2}(x)\mathbf{a}^{\intercal}\mathbf{\Sigma}\mathbf{a}x^{2}M\right)\right\} .\label{eq:Hamiltonian2}
\end{equation}
All the assumptions in \citet{fleming1989existence} are satisfied,
so $\bar{v}^{k}(t,x)$ \eqref{eq:value2} is a viscosity solution
of the HJBI equation \eqref{eq:HJBI2}.

Now we introduce another value function 
\[
\tilde{v}^{k}(t,x)=\begin{cases}
\bar{v}^{k}(t,x) & \forall(t,x)\in[0,T]\times B_{k+1}\\
U(x)+\inf_{\Sigma\in\mathcal{B}}\mathbb{E}^{t,x}[\lambda_{0}\int_{t}^{T}F(\Sigma_{s})ds] & \forall(t,x)\in[0,T]\times(\mathbb{R}\backslash B_{k+1})
\end{cases}.
\]
In the first case where $x\in B_{k+1}$, we have $\left(X_{T}^{k}\right)^{2}<\left(k+2\right)^{2}$
almost surely. Therefore 
\[
\tilde{v}^{k}(t,x)=\adjustlimits\sup_{\Gamma\in\mathcal{N}}\inf_{\Sigma\in\mathcal{B}}\mathbb{E}^{t,x}[\lambda_{0}\int_{t}^{T}F(\Sigma_{s})ds+U(X_{T}^{k,{\scriptscriptstyle \Gamma,\Sigma}})],\:\:\forall(t,x)\in[0,T]\times B_{k+1}.
\]
Then $\tilde{v}^{k}(t,x),\forall(t,x)\in[0,T]\times B_{k+1}$ is a
viscosity solution of 
\begin{align}
\begin{cases}
\frac{\partial v}{\partial t}(t,x)+H^{k}(t,x,\frac{\partial v}{\partial x}(t,x),\frac{\partial^{2}v}{\partial x^{2}}(t,x))=0 & \text{in }\,[0,T)\times\mathbb{R},\\
v(T,x)=U(x) & \text{on }\,[T]\times\mathbb{R}.
\end{cases}\label{eq:HJBI3}
\end{align}
Since the drift and diffusion of $X_{t}^{k}$ are zero outside of
$B_{k+1}$, then $X_{T}^{k,{\scriptscriptstyle \Gamma,\Sigma}}=x$
for $x\in(\mathbb{R}\backslash B_{k+1})$ and 
\begin{alignat*}{1}
\tilde{v}^{k}(t,x) & =\sup_{\Gamma\in\mathcal{N}}\inf_{\Sigma\in\mathcal{B}}\mathbb{E}^{t,x}[\lambda_{0}\int_{t}^{T}F(\Sigma_{s})ds+U(X_{T}^{k,{\scriptscriptstyle \Gamma,\Sigma}})],\:\:\forall(t,x)\in[0,T]\times\left(\mathbb{R}\backslash B_{k+1}\right).
\end{alignat*}
It is easy to check that $\tilde{v}^{k}(t,x),\forall(t,x)\in[0,T]\times(\mathbb{R}\backslash B_{k+1})$
is also a viscosity solution of HJBI equation \eqref{eq:HJBI3} with
$\phi_{k}(x)=0$. Combining the two cases, we have 
\[
\tilde{v}^{k}(t,x)=\sup_{\Gamma\in\mathcal{N}}\inf_{\Sigma\in\mathcal{B}}\mathbb{E}^{t,x}[\lambda_{0}\int_{t}^{T}F(\Sigma_{s})ds+U(X_{T}^{k,{\scriptscriptstyle \Gamma,\Sigma}})],\quad\text{on }[0,T]\times\mathbb{R},
\]
and $\tilde{v}^{k}(t,x)$ is a viscosity solution of \eqref{eq:HJBI3}.

Since $H^{k}$ $\rightarrow H$ as $k\rightarrow\infty$, if we can
prove $\tilde{v}^{k}\rightarrow\bar{v}$ as $k\rightarrow\infty$,
then it shows that $\bar{v}$ is a viscosity solution of equation
\eqref{eq:HJBI1}. We will prove the convergence in the following
way: first of all, we have 
\begin{alignat*}{1}
\Bigl|\bar{v}-\tilde{v}^{k}\Bigr| & \leq\sup_{\Gamma}\sup_{\Sigma}\biggl|\mathbb{E}^{t,x}\Bigl[\lambda_{0}\int_{t}^{T}F(\Sigma_{s})ds+U(X_{T}^{{\scriptscriptstyle \Gamma,\Sigma}})\Bigr]-\mathbb{E}^{t,x}\Bigl[\lambda_{0}\int_{t}^{T}F(\Sigma_{s})ds+U(X_{T}^{k,{\scriptscriptstyle \Gamma,\Sigma}})\Bigr]\biggr|.
\end{alignat*}

For any arbitrary pair of control and strategy processes $(\Gamma,\Sigma)$,
we have 
\begin{equation}
\mathbb{E}^{t,x}\Bigl[\Big(\bigl(\lambda_{0}\int_{t}^{T}F(\Sigma_{s})ds+U(X_{T}^{{\scriptscriptstyle \Gamma,\Sigma}})\bigr)-\bigl(\lambda_{0}\int_{t}^{T}F(\Sigma_{s})ds+U(X_{T}^{k,{\scriptscriptstyle \Gamma,\Sigma}})\bigr)\Big)\mathbbm{1}(\tau_{k}\geq T)\Bigr]=0.\label{eq:vis solution before stopping time}
\end{equation}
Using Assumption \prettyref{assu:U}, we can write 
\begin{alignat}{1}
 & \biggl|\mathbb{E}^{t,x}[\lambda_{0}\int_{t}^{T}F(\Sigma_{s})ds+U(X_{T}^{{\scriptscriptstyle \Gamma,\Sigma}}])]-\mathbb{E}^{t,x}[\lambda_{0}\int_{t}^{T}F(\Sigma_{s})ds+U(X_{T}^{k,{\scriptscriptstyle \Gamma,\Sigma}})]\biggr|\nonumber \\
={} & \biggl|\mathbb{E}^{t,x}\Bigl[\bigl(U(X_{T}^{{\scriptscriptstyle \Gamma,\Sigma}})-U(X_{T}^{k,{\scriptscriptstyle \Gamma,\Sigma}})\bigr)\mathbbm{1}(\tau_{k}<T)\Bigr]\biggr|\nonumber \\
\leq{} & \mathbb{E}^{t,x}\Bigl[Q(\bigl|X_{T}^{{\scriptscriptstyle \Gamma,\Sigma}}\bigr|,\bigl|X_{T}^{k,{\scriptscriptstyle \Gamma,\Sigma}}\bigr|)(\bigl|X_{T}^{{\scriptscriptstyle \Gamma,\Sigma}}\bigr|-\bigl|X_{T}^{k,{\scriptscriptstyle \Gamma,\Sigma}}\bigr|)\mathbbm{1}(\tau_{k}<T)\Bigr].\label{eq:polynomial indicator}
\end{alignat}
Applying the Cauchy-Schwarz inequality on the upper bound \eqref{eq:polynomial indicator},
with similar arguments in \eqref{Markov}, we obtain 
\begin{alignat}{1}
 & \left(\mathbb{E}^{t,x}\Bigl[Q(\left|X_{T}^{{\scriptscriptstyle \Gamma,\Sigma}}\right|,\bigl|X_{T}^{k,{\scriptscriptstyle \Gamma,\Sigma}}\bigr|)(\bigl|X_{T}^{{\scriptscriptstyle \Gamma,\Sigma}}\bigr|-\bigl|X_{T}^{k,{\scriptscriptstyle \Gamma,\Sigma}}\bigr|)\mathbb{I}(\tau_{k}<T)\Bigr]\right){}^{2}\nonumber \\
\leq{} & C_{T}\left(1+x^{2m}\right)\times\frac{C_{T}\left(1+x^{2}\right)}{k^{2}}.\label{eq:Markov and SDE}
\end{alignat}
Hence 
\begin{equation}
\mathbb{E}^{t,x}\Bigl[Q(\left|X_{T}^{{\scriptscriptstyle \Gamma,\Sigma}}\right|,\bigl|X_{T}^{k,{\scriptscriptstyle \Gamma,\Sigma}}\bigr|)(\bigl|X_{T}^{{\scriptscriptstyle \Gamma,\Sigma}}\bigr|-\bigl|X_{T}^{k,{\scriptscriptstyle \Gamma,\Sigma}}\bigr|)\mathbbm{1}(\tau_{k}<T)\Bigr]\leq\frac{\Phi(\left|x\right|)}{k},\label{eq:HJBI estimate}
\end{equation}
where $\Phi(\left|x\right|)$ is a polynomial function independent
of $k$. Since $(\Gamma,\Sigma)$ are arbitrary, combining \eqref{eq:vis solution before stopping time},
\eqref{eq:polynomial indicator} and \eqref{eq:HJBI estimate}, we
deduce that 
\begin{alignat*}{1}
\Bigl|\bar{v}-\tilde{v}^{k}\Bigr| & \leq\frac{\Phi(\left|x\right|)}{k}.
\end{alignat*}
So $\tilde{v}^{k}$ converges to $\bar{v}$ as $k\rightarrow\infty$.
Thus $\bar{v}$ is a viscosity solution of the HJBI equation \eqref{eq:HJBI1}. 
\end{proof}

\subsection{Explicit solution of equation \eqref{eq:quartic}\label{sec:Appendix-explicit solution}}

For completeness, we express the real positive root of equation
\eqref{eq:quartic} explicitly.

Let $c=\dfrac{(\mu-r)^{2}}{2\lambda_{0}}$, the discriminant of the
equation $\Delta=-256c^{3}-27\sigma_{0}^{4}c^{2}$ is less than zero,
meaning there are two distinct real roots. It is easy to check that
there is one positive and one negative root, and the real positive
one is

\begin{gather*}
\hat{\sigma}=\frac{\sigma_{0}}{4}+\frac{1}{2}\left[\frac{\sigma_{0}^{2}}{4}+\frac{\sqrt[3]{\sqrt{3}\sqrt{27\sigma_{0}^{4}c^{2}+256c^{3}}-9\sigma_{0}^{2}c}}{\sqrt[3]{2}3^{2/3}}-\frac{4\sqrt[3]{\frac{2}{3}}c}{\sqrt[3]{\sqrt{3}\sqrt{27\sigma_{0}^{4}c^{2}+256c^{3}}-9\sigma_{0}^{2}c}}\right]^{\frac{1}{2}}\\
+\frac{1}{2}\Biggl[\frac{\sigma_{0}^{2}}{2}-\frac{\sqrt[3]{\sqrt{3}\sqrt{27\sigma_{0}^{4}c^{2}+256c^{3}}-9\sigma_{0}^{2}c}}{\sqrt[3]{2}3^{2/3}}+\frac{4\sqrt[3]{\frac{2}{3}}c}{\sqrt[3]{\sqrt{3}\sqrt{27\sigma_{0}^{4}c^{2}+256c^{3}}-9\sigma_{0}^{2}c}}\\
+\frac{\sigma_{0}^{3}}{4\sqrt{\frac{\sigma_{0}^{2}}{4}+\frac{\sqrt[3]{\sqrt{3}\sqrt{27\sigma_{0}^{4}c^{2}+256c^{3}}-9\sigma_{0}^{2}c}}{\sqrt[3]{2}3^{2/3}}-\frac{4\sqrt[3]{\frac{2}{3}}c}{\sqrt[3]{\sqrt{3}\sqrt{27\sigma_{0}^{4}c^{2}+256c^{3}}-9\sigma_{0}^{2}c}}}}\Biggr]^{\frac{1}{2}}.
\end{gather*}

\end{document}